\documentclass[a4paper]{amsart}
\usepackage[french, UKenglish]{babel}

\usepackage{amssymb,latexsym}

\author[W.\thinspace{}Kim]{Wansu Kim}
\address{Wansu Kim\\%
Korea Institute for Advanced Study (KIAS) \\
Department of Mathematics\\%
Seoul, 02455\\%
South Korea}
\email{wansukim@kias.re.kr}

\usepackage[active]{srcltx}
\usepackage{amsmath}
\usepackage{amsfonts}

\usepackage{amsthm}
\usepackage{amscd}
\usepackage{mathrsfs}

\usepackage{fancyhdr}
\usepackage{graphicx}
\usepackage{psfrag}
\usepackage{calc}
\usepackage{url}

\usepackage{pb-diagram,pb-xy}

\usepackage{times}
\usepackage{verbatim}
\usepackage[cmtip,arrow,matrix,curve,tips,frame]{xy}
\usepackage{hyperref}
\xyoption{matrix}

\usepackage{ifthen}
\usepackage{color}
\usepackage{amsxtra}
\usepackage{amstext}
\usepackage{amssymb}
\usepackage{stmaryrd}
\usepackage{mathrsfs}
\usepackage{epsfig}
\usepackage{pifont}
\usepackage{charter}
\usepackage{avant}
\usepackage{courier}
\usepackage[T1]{fontenc}
\usepackage{textcomp}
\usepackage{bbm}

\numberwithin{equation}{subsection}

\theoremstyle{plain}
\newtheorem{thm}[subsection]{Theorem}
\newtheorem*{thm*}{Theorem}

\newtheorem{thmsub}[equation]{Theorem}

\newtheorem*{exthm*}{Expected Theorem}

\newtheorem{lemsub}[equation]{Lemma}
\newtheorem*{lem*}{Lemma}

\newtheorem{prop}[subsection]{Proposition}

\newtheorem{propsub}[equation]{Proposition}
\newtheorem*{prop*}{Proposition}

\newtheorem{corsub}[equation]{Corollary}
\newtheorem*{cor*}{Corollary}

\newtheorem*{claim*}{Claim}

\newtheorem*{conj*}{Conjecture}

\theoremstyle{definition}

\newtheorem*{defn*}{Definition}
\newtheorem{defnsub}[equation]{Definition}

\newtheorem{exasub}[equation]{Example}
\newtheorem*{exa*}{Example}

\theoremstyle{remark}

\newtheorem{rmksub}[equation]{Remark}
\newtheorem*{rmk*}{Remark}

\numberwithin{figure}{subsection}
\numberwithin{table}{subsection}

\newcounter{listnum}

\DeclareMathOperator{\Adm}{Adm}

\newcommand{\BX}{\mathbb{X}}
\newcommand{\BY}{\mathbb{Y}}

\DeclareMathOperator{\pr}{pr}

\DeclareMathOperator{\id}{id}

\newcommand{\dcj}{\text{-}\mathrm{conj}}

\DeclareFontEncoding{OT2}{}{} 

\newcommand{\ol}[1]{\overline{#1}}
\newcommand{\wt}[1]{\widetilde{#1}}

\newcommand{\et}{\text{\rm\'et}}

\newcommand{\wh}[1]{\widehat{#1}}

\newcommand{\XX}{\mathfrak{X}}

\newcommand{\sC}{\mathscr{C}}

\newcommand{\cH}{\mathcal{H}}

\newcommand{\cG}{\mathcal{G}}

\newcommand{\cO}{\mathcal{O}}

\newcommand{\bM}{\boldsymbol{\mathrm{M}}}

\newcommand{\triv}{\mathbf{1}}

\newcommand{\llpar}{(\!(}
\newcommand{\rrpar}{)\!)}

\DeclareMathOperator{\Spec}{Spec}

\DeclareMathOperator{\Spf}{Spf}

\newcommand{\perf}{\mathrm{perf}}

\newcommand{\N}{\mathcal{N}}

\DeclareMathOperator{\Frac}{Frac}

\newcommand{\ra}{\rightarrow}

\newcommand{\hra}{\hookrightarrow}

\newcommand{\thra}{\twoheadrightarrow}

\newcommand{\riso}{\xrightarrow{\sim}}

\newcommand{\dra}{\dashrightarrow}

\newcommand{\com}[1]{^{(#1)}}

\newcommand{\Gm}{\mathbb{G}_m}

\newcommand{\set}[1]{\{#1\}}

\newcommand{\iv}{^{-1}}
\newcommand{\ivtd}[1]{\tfrac{1}{#1}}

\newcommand{\sig}{\sigma}

\newcommand{\Q}{\mathbb{Q}}

\newcommand{\cS}{\mathcal{S}}

\newcommand{\KK}{{\mathsf K}}

\newcommand{\Qbar}{\overline{\Q}}

\newcommand{\cT}{\mathcal{T}}

\newcommand{\Z}{\mathbb{Z}}

\newcommand{\F}{\mathbb{F}}

\newcommand{\Fpbar}{\overline{\F}_p}

\newcommand{\Qp}{\Q_p}
\newcommand{\Qpbr}{\breve\Q_p}

\newcommand{\Qpbar}{\Qbar_p}

\newcommand{\Zp}{\Z_p}
\newcommand{\Zpbr}{\breve\Z_p}

\newcommand{\Fp}{\F_p}
\newcommand{\Fq}{\F_q}

\newcommand{\gl}{\mathfrak{gl}}
\DeclareMathOperator{\Nilp}{Nilp}

\DeclareMathOperator{\Isom}{Isom}

\newcommand{\loc}{\mathrm{loc}}

\newcommand{\OO}{\mathcal{O}}
\newcommand{\fo}{\mathscr{O}}

\newcommand{\sA}{\mathscr{A}}

\DeclareMathOperator{\Aut}{Aut}
\DeclareMathOperator{\diag}{diag}

\DeclareMathOperator{\Fr}{Fr}

\newcommand{\GL}{\mathrm{GL}}
\DeclareMathOperator{\GSp}{GSp}

\newcommand{\der}{\mathrm{der}}
\newcommand{\ad}{\mathrm{ad}}

\DeclareMathOperator{\Gal}{Gal}

\newcommand{\cB}{\mathcal{B}}

\newcommand{\Zink}{\mathbb{W}}

\newcommand{\g}{\mathfrak{g}}

\newcommand{\m}{\mathfrak{m}}

\newcommand{\gC}{\mathfrak{C}}

\newcommand{\gJ}{\mathfrak{J}}

\newcommand{\gN}{\mathfrak{N}}

\newcommand{\GLn}{\GL_n}

\newcommand{\ur}{\mathrm{ur}}
\newcommand{\univ}{\mathrm{univ}}

\DeclareMathOperator{\Hom}{Hom}

\newcommand{\lhom}{\mathscr{H}\mathit{om}}

\DeclareMathOperator{\End}{End}

\DeclareMathOperator{\Lie}{Lie}

\newcommand{\BI}{\mathbb{I}}

\newcommand{\cZ}{\mathcal{Z}}

\newcommand{\gr}{\mathtt{gr}}

\newcommand{\UH}{\mathfrak{H}}
\newcommand{\sS}{\mathscr{S}}

\DeclareMathOperator{\isoc}{isoc}
\newcommand{\cris}{\mathrm{cris}}

\newcommand{\Acris}{A_{\cris}}

\newcommand{\Bcris}{B_{\cris}}

\DeclareMathOperator{\Rep}{Rep}

\DeclareMathOperator{\Qisg}{Qisg}

\newcommand{\DD}{\mathbb{D}}
\newcommand{\bD}{\mathbf{D}}

\newcommand{\db}[1]{[\![#1]\!]}

\title[On central leaves of Hodge-type Shimura varieties]{On central leaves of Hodge-type Shimura varieties with parahoric level structure}

\begin{document}

\begin{abstract}
Kisin and Pappas \cite{KisinPappas:ParahoricIntModel} constructed integral models of Hodge-type Shimura varieties with parahoric level structure at $p>2$, such that the formal neighbourhood of a mod~$p$ point can be interpreted as a deformation space of $p$-divisible group with some Tate cycles (generalising Faltings' construction).
In this paper, we study the central leaf and the closed Newton stratum in the formal neighbourhoods of mod~$p$ points of Kisin-Pappas integral models with parahoric level structure; namely, we obtain the dimension of central leaves and the almost product structure of Newton strata.
In the case of hyperspecial level strucure (i.e., in the good reduction case), our main results were already obtained by Hamacher \cite{Hamacher:DeforSpProdStr}, and the result of this paper holds for ramified groups as well.
\end{abstract}
\keywords{Shimura varieties of Hodge type, deformations of $p$-divisible groups}
\subjclass[2010]{14L05, 14G35}
\maketitle
\tableofcontents

\section{Introduction}
Let $\sS$ be a moduli space over $\Z_{(p)}$ of principally polarised $g$-dimensional abelian varieties with some prime-to-$p$ level structure, and let $\sA$ be the universal abelian scheme over $\sS$. Given a geometric closed point $x\in\sS(\Fpbar)$, Oort  found a smooth equi-dimensional locally closed subscheme $\sC(x)\subset \sS_{\Fpbar}$ containing $x$, which is the locus where  the  $p$-divisible group associated to the universal abelian scheme is ``fibrewise constant'' (i.e., the geometric fibres of the $p$-divisible group are all isomorphic to $\sA_x[p^\infty]$), and its dimension can be explicitly computed in terms of the Newton polygon of $\sA_x[p^\infty]$; \emph{cf.} \cite[Theorems~2.2,~3.13]{Oort:Foliations}, \cite[\S7]{Chai:HeckeSiegel}.
Such $\sC(x)$ is called a \emph{central leaf}.

Oort also showed that by transporting $\sC(x)$ by ``isogeny correspondences'', one can ``fill up'' the Newton stratum of $\sS$ that contains $x$. A stronger and more precise statement can be formulated as the ``almost product structure'' of Newton strata; \emph{cf.} \cite[Theorem~5.3]{Oort:Foliations}.

The original motivation of Oort's study of central leaves \cite{Oort:Foliations}  is to understand ``Hecke orbits'' in $\sS_{\Fpbar}$. Later, Mantovan \cite{Mantovan:Thesis,Mantovan:Foliation} found an interesting application of the PEL generalisation of the almost product structure to the study of the cohomology of compact  PEL Shimura varieties at hyperspecial level at $p$, which is often referred to as \emph{Mantovan's formula}. (Roughly speaking, Mantovan's formula is the cohomological consequence of the almost product structure, if it is interpreted in terms of ``Igusa towers'' over central leaves and Rapoport-Zink spaces.)

Now, let us consider an integral canonical model $\sS$ of a Hodge-type Shimura variety at \emph{hyperspecial} level at $p>2$. Then generalising the definitions of Igusa towers and Rapoport-Zink spaces (together with all the expected group actions)  turns out to be not so trivial in this setting, because $\sS$ does not carry a convenient moduli interpretation, unlike the PEL case. (See \cite{HowardPappas:GSpin} for the construction of Hodge-type Rapoport-Zink spaces and the relationship to Hodge-type Shimura varieties, which is a simplification of the original construction in \cite{Kim:RZ,Kim:Unif}. For Hodge-type Igusa varieties, see \cite{Hamacher:ShVarProdStr}.) With these basic definitions in place, Hamacher \cite{Hamacher:ShVarProdStr} was able to prove the unramified Hodge-type generalisation of the almost product structure.

 It is worth noting that  these basic constructions for the unramified Hodge-type case crucially uses Kisin's study of isogeny classes of mod~$p$ points of the integral canonical models; namely, the existence of a natural map from certain affine Deligne-Lusztig varieties to the set of mod~$p$ points of the integral canonical models (\emph{cf.} \cite[Proposition~1.4.4]{Kisin:LanglandsRapoport}), which is a highly non-trivial result unlike the PEL case. Also, the strategy of proving the almost product structure for unramified Hodge-type Shimura varieties is to study the closed Newton stratum in the deformation space of $p$-divisible groups with Tate cycles \cite[Proposition~4.6]{Hamacher:DeforSpProdStr}, and globalise it using the result on isogeny classes of mod~$p$ points of Shimura varieties \cite[Proposition~1.4.4]{Kisin:LanglandsRapoport}.

Now, let $(G,\set{h})$ denote a Hodge-type Shimura datum, and let $p$ be an odd prime such that $p$ does not divide the order of $\pi_1(G^\der)$ and $G_{\Qp}$ splits after a tame extension (but not necessarily unramified). Then Kisin and Pappas \cite{KisinPappas:ParahoricIntModel} constructed integral models $\sS$ of Shimura varieties for $(G,\set{h})$ with parahoric level structure at $p$, whose formal completions at closed points are isomorphic to the formal completions of  the local models constructed by Pappas and Zhu \cite{PappasZhu:LocMod}. Although this integral model does not carry any convenient moduli interpretation, the formal completions of $\sS$ at closed points have a nice interpretation as the deformation spaces of $p$-divisible groups with certain cycles (generalising the case of hyperspecial levels); \emph{cf.} \cite[Corollary~4.2.4]{KisinPappas:ParahoricIntModel}.

Along the way, Kisin and Pappas defined a ``universal deformation'' of $p$-divisible groups with Tate cycles over the completed local ring of certain Pappas-Zhu local models (\emph{cf.} \cite[\S3]{KisinPappas:ParahoricIntModel}), which can be viewed as a generalisation of Faltings' construction in the unramified case (\emph{cf.} \cite[\S7]{Faltings:IntegralCrysCohoVeryRamBase}, \cite[\S4]{Moonen:IntModels}). The goal of this paper is to generalise Hamacher's results \cite{Hamacher:DeforSpProdStr} on central leaves and the closed Newton strata of Faltings deformation spaces to the ``Kisin-Pappas deformation spaces'' (which can be defined for certain ramified groups).

Let us now explain our deformation-theoretic result under the following global setting (which is the motivating case). Let $(G,\set h)$ denote a Hodge-type Shimura datum, and choose $p>2$ such that $G_{\Qp}$ splits after a tame extension. Assume that $p$ does not divide the order of $\pi_1(G^\der)$. Let $\sS$ be an integral model of a Shimura variety for $(G,\set h)$ with parahoric level structure at $p$ constructed by Kisin and Pappas, and let $\sA$ be the principally polarised abelian scheme over $\sS$ corresponding to a finite unramified map from $\sS$ into some Siegel modular variety inducing a closed immersion on the generic fibres. Let $\bar\sS$ denote the base change of the special fibre of $\sS$ over $\Fpbar$.

Let $x\in\bar\sS(\Fpbar)$. Then over the perfection $(\wh\cO_{\bar\sS,x})^{p^{-\infty}}$ of $\wh\cO_{\bar\sS,x}$ we have an $F$-isocrystal with $G$-structure in the sense of \cite[Definition~3.3]{RapoportRichartz:Gisoc} coming from the relative crystalline homology of $\sA$ and the ``crystalline realisation'' of absolute Hodge cycles on the generic fibre of $\sA$; \emph{cf.} Lemma~\ref{lem:GIsoc}. In particular, we can associate to $x$ a $\sig$-$G(\Qpbr)$ conjugacy class $[b]$, and it is possible to define the Newton stratification on $\Spec\wh\cO_{\bar\sS,x}$. Let $\gN_\cG\subset\Spec\wh\cO_{\bar\sS,x}$ denote the closed Newton stratum.

Let us now ``preview'' the main results of this paper. (We refer to the cited theorems for the precise statement.)
\begin{thm}\label{thm:main}
Let $x\in\bar\sS(\Fpbar)$, and consider the reduced locally closed subscheme $\sC:=\sC(x)\subset \bar\sS$ whose geometric points $\bar y$ are exactly those such that $\sA_{\bar y}[p^\infty]$ is isomorphic to $\sA_x[p^\infty]$. Then we have the following:
\begin{enumerate}
\item\label{thm:main:CenLeaf} (Corollary~\ref{cor:CenLeaf}) $\sC$ is smooth of equidimension $\langle 2\rho,\nu_{[b]}\rangle$, where $2\rho$ is the sum of positive roots of $G$ and $\nu_{[b]}$ is defined in Proposition~\ref{prop:QEnd}.
\item\label{thm:main:APS} (Corollary~\ref{cor:APS}) Let $\gC_\cG:=\Spec\wh\cO_{\sC,x}\subset\Spec\wh\cO_{\bar\sS,x}$, and let $\gJ_\cG\subset\Spec\wh\cO_{\bar\sS,x}$ denote the ``isogeny leaf'' (\emph{cf.} the paragraph above Theorem~\ref{thm:APS}). Then we have a natural isomorphism of perfect schemes
\[\pi_\infty:\gC_\cG^{p^{-\infty}}\times\gJ_\cG^{p^{-\infty}} \riso \gN_\cG^{p^{-\infty}},
\]
compatible with the case when $\cG = \GL(\Lambda)$ in  \cite[Corollary~4.4]{Hamacher:DeforSpProdStr}.
\end{enumerate}
\end{thm}

If $\sS$ is a Siegel modular variety, then Theorem~\ref{thm:main}(\ref{thm:main:APS}) is a consequence of the ``almost product strurcure of the Newton strata'' \cite[Theorem~5.3]{Oort:Foliations}, and there are number of different proofs of Theorem~\ref{thm:main}(\ref{thm:main:CenLeaf}), for which we refer to \cite{Oort:DimLeaves} and references therein. Some proofs for the Siegel case can be generalised to the case of PEL Shimura varieties at hyperspecial level at~$p$. If $\sS$ is an integral canonical model of Hodge-type Shimura varieties with hyperspecial level at $p>2$, then this theorem was obtained by Hamacher \cite{Hamacher:DeforSpProdStr}. On the other hand, the author is not aware if the dimension of $\sC$ was obtained in the ramified PEL case in the literature. 

There is a parallel story for the local field of characteristic~$p$, where local shtukas play the role of $p$-divisible groups. Under some additional hypothesis on $\cG$, the analogue of Theorem~\ref{thm:main} for the deformation space of $\cG$-shtukas is already available; for example, \cite{HartlViehmann:Foliation} when $\cG$ is a split reductive group over $\Fq[[z]]$, and Viehmann and Wu \cite{ViehmannWu:CenLeaf} when $\cG$ is unramified (i.e., reductive over $\Fq[[z]]$). The global function field analogue of Theorem~\ref{thm:main} (for the function field analogue of Shimura varieties) was also obtained by Neupert \cite[Main~Theorem~1]{Neupert:thesis} under a certain restrictions on the group. It is quite plausible that these result could be generalised to allow certain ramification on the group $\cG$.

Let us remark on the proof of Theorem~\ref{thm:main}. 
Let $\BX:=\sA_x[p^\infty]$, equipped with the crystalline Tate cycles $(s_\alpha)$ (which come from the absolute Hodge cycles of the generic fibre of $\sA$). The key construction in the proof is the following:
\begin{prop}[{Corollary~\ref{cor:TenHom}, Proposition~\ref{prop:Qisg}}]\label{prop:main}
Assume that the $p$-divisible group $\BX$ is completely slope divisible (\emph{cf.} Definition~\ref{def:SlopeDivBT}).
Then there exists a formal group scheme $\Qisg_G(\BX)$, flat over $\Zpbr$ with perfect special fibre, which satisfies the following property:
\begin{enumerate}
\item For any f-semiperfect ring $R$, $\Qisg_G(\BX)(R)$ is naturally isomorphic to the group of self quasi-isogenies of $\BX_R$ preserving the tensors $(s_\alpha)$ in the sense of Definition~\ref{def:TenHom}. 
\item There is a natural isomorphism $\Qisg_G(\BX) \cong \Qisg_G^\circ(\BX)\rtimes \underline{J_b(\Qp)}$, where $\Qisg_G^\circ(\BX) \cong \Spf\Zpbr[[x_1^{p^{-\infty}},\cdots,x_d^{p^{-\infty}}]]$ as a formal scheme. Here, $d:=\langle 2\rho,\nu_{[b]}\rangle$ as in Theorem~\ref{thm:main}(\ref{thm:main:CenLeaf}).
\end{enumerate}
\end{prop}
In the unramified PEL case, the formal group scheme $\Qisg_G(\BX)$ was constructed in the work of Caraiani and Scholze   \cite[\S4.2]{CaraianiScholze:ShVar}, which was denoted as $\Aut_G(\wt \BX)$ in \emph{loc.~cit.} 
Clearly, this proposition can be extended to a $p$-divisible groups with tensors admitting a tensor-preserving quasi-isogeny to some completely slope divisible $p$-divisible groups. See Proposition~\ref{prop:EmbADLV} for a group-theoretic formulation of this condition.

It turns out that the connected component $\Qisg^\circ_G(\BX)$ of $\Qisg_G(\BX)$ acts on the formal completion $\wh{\sS}_x$; \emph{cf.} Theorem~\ref{thm:QisgActionRZ}. We obtain Theorem~\ref{thm:main} by interpreting the central leaf and the closed Newton stratum in $\Spec\wh\cO_{\bar\sS,x}$ in terms of the $\Qisg^\circ_G(\BX)$-orbits of the closed point and the isogeny leaf, respectively; \emph{cf.} Theorems~\ref{thm:CenLeaf},~\ref{thm:APS}. (In the unramified PEL case, this strategy to obtain the dimension of the central leaves was mentioned in \cite[Remark~4.2.13]{CaraianiScholze:ShVar}.) It seems that the $\Qisg^\circ_G(\BX)$-action on $\wh{\sS}_x$ could be of separate interest; \emph{cf.} Remark~\ref{rmk:QisgActionRZ}.

Our motivation for Theorem~\ref{thm:main}, especially (\ref{thm:main:APS}), is to generalise the almost product structure of Newton strata in the integral models of Hodge-type Shimura varieties at odd ``tame'' primes, generalising the result of Hamacher's \cite{Hamacher:ShVarProdStr}. As alluded earlier, however, in order to define Rapoport-Zink spaces and Igusa varieties we need to have some control of each isogeny class of mod~$p$ points of $\sS$ in terms of some union of affine Deligne-Lusztig varieties. Although this desired result is not known in the full generality, there are some cases where we can obtain it; \emph{cf.} \cite[Proposition~1.4.4]{Kisin:LanglandsRapoport}, \cite[Theorem~0.2]{HeZhou:ConnCompADL}. (See Remark~\ref{rmk:globalAPS} for more discussions.)
In the joint work with Hamacher \cite{HamacherKim:Mantovan}, we used Theorem~\ref{thm:main} to generalise the almost product structure \cite{Hamacher:ShVarProdStr} (allowing $G_{\Qp}$ to be ``tamely ramified'') and try to study the cohomological consequence, assuming a certain natural conjecture on isogeny classes of mod~$p$ points of Hodge-type Shimura varieties.

In \S\ref{sec:Notations} we review some group-theoretic and (semi-)linear algebraic background. In~\S\ref{sec:Qisg} we introduce the group of tensor-preserving self quasi-isogenies $\Qisg_G(\BX)$ and prove some basic properties stated in Proposition~\ref{prop:main}. In \S\ref{sec:KP}, we review the ``Kisin-Pappas deformation theory'' \cite[\S3]{KisinPappas:ParahoricIntModel}, and show that the connected component of the tensor-preserving self quasi-isogeny group $\Qisg_G(\BX)$ acts on ``Kisin-Pappas deformation spaces''; \emph{cf.} Theorem~\ref{thm:QisgActionRZ}. In \S\ref{sec:APS} we prove Theorem~\ref{thm:main}.

\subsection*{Acknowledgement}
The author thanks Paul Hamacher for many helpful discussions, and Julien Hauseux for his help with group theory used in the proof of Proposition~\ref{prop:QEnd}. The author would also like to thank the anonymous referee whose comments were greatly helpful. This work was supported by the EPSRC (Engineering and Physical Sciences Research Council) in the form of EP/L025302/1.

\section{Notation and preliminaries}\label{sec:Notations}
For a finite extension $E$ over $\Qp$, we write $\breve E:=\wh E^\ur$.

\subsection{Review of parahoric groups and extended affine Weyl groups}\label{subsec:GpTh}
\begin{defnsub}\label{def:KP}
Let $G$ be a connected reductive group over $\Qp$, and we fix a point $x\in\cB(G,\Qp)$ in the extended Bruhat-Tits building. For any algebraic field extension  $K/\Qp$ that is finitely ramified, we view $x\in\cB(G,\Qp)$ via the natural embedding $\cB(G,\Qp)\hra\cB(G,K)$. Then we obtain the following smooth $\Zp$-models of $G$:
\begin{enumerate}
\item The \emph{Bruhat-Tits group scheme} $\cG ( = \cG_x)$, which has the property that $\cG_x(\Zp^\ur)\subset G(\Qp^\ur)$ is the stabiliser of $x\in\cB(G,\Qp^\ur)$.
\item The \emph{parahoric group scheme} $\cG^\circ ( = \cG^\circ_x)$, which is the open subgroup of $\cG$ with connected special fibre.
\end{enumerate}
Note that $\cG(\Zp)\subset G(\Qp)$ is the full stabiliser of $x\in\cB(G,\Qp)$, as well as the full stabiliser of the interior of the facet containing $x$.

We write $\KK:=\cG(\Zp)$, $\KK^\circ:=\cG^\circ(\Zp)$, $\breve\KK:=\cG(\Zpbr)$, and $\breve\KK^\circ:=\cG^\circ(\Zpbr)$.
\end{defnsub}

For any connected reductive group $G$ over $\Qpbr$, Kottwitz \cite[(7.1.1)]{Kottwitz:Gisoc2} defined a natural surjective homomorphism
\begin{equation}\label{eqn:KottHom}
\kappa_G:G(\Qpbr)\thra \pi_1(G)_{I_{\Qp}} ( = X^*(Z(\wh G)^{I_{\Qp}})),
\end{equation}
which is functorial in $G$, where $\pi_1(G)$ is the algebraic fundamental group and $I_{\Qp}$ is the inertia group of $\Qp$.
Note that $\kappa_G$ can be concretely described when $G$ is a torus (\emph{cf.} \cite[\S7.2]{Kottwitz:Gisoc2}). Since $\kappa_G$ is required to be functorial in $G$ and $\kappa_G$ should be trivial if $G$ is simply connected and semi-simple (as $\pi_1(G) = 0$), the torus case uniquely determines $\kappa_G$ for any connected reductive group $G$ over $\Qpbr$ (\emph{cf.} \cite[\S7.3]{Kottwitz:Gisoc2}).

We set
\begin{equation}\label{eqn:ParahoricGenerated}
 G(\Qpbr)_1:=\ker(\kappa_G).
\end{equation}

Returning to the setting of Definition~\ref{def:KP} (so $G$ is now a connected reductive group over $\Qp$),
let us recall some basic facts:
\begin{enumerate}
\item We have $\KK^\circ = \KK\cap G(\Qpbr)_1$ and $\breve\KK^\circ = \breve\KK\cap G(\Qpbr)_1$; \emph{cf.} \cite[Appendix, Proposition~3, Remark~11]{HainesRapoport:Gamma1p-Drinfeld}. In other words, $\KK^\circ$ and $\breve\KK^\circ$ are parahoric subgroups in the sense of  \cite[Appendix, Definition~1]{HainesRapoport:Gamma1p-Drinfeld}.
\item If $G=T$ is a torus, then we have $T(\Qpbr)_1 = \mathcal{T}^\circ(\Zpbr)$, where $\mathcal{T}^\circ$ is the connected  N\'eron model of $T$. It turns out that $T(\Qpbr)_1$ is the unique parahoric subgroup of $T(\Qpbr)$, and
we have the following short exact sequence
\[
\xymatrix@1{
1 \ar[r] & T(\Qpbr)_1 \ar[r] & T(\Qpbr)\ar[r]^-{\kappa_T} & X_*(T)_{I_{\Qp}} \ar[r] &0.
}
\]
 In general, $G(\Qpbr)_1$ is generated by the parahoric subgroups of $G(\Qpbr)$.
\end{enumerate}

Let us now make a convenient choice of maximal torus $T\subset G$ as follows.
Let $S$ be a maximal $\Qpbr$-split torus of $G$ which is defined over $\Qp$ and contains a maximal $\Qp$-split torus. We furthermore arrange the choice of  $S$ so that the maximal $\Qp$-split subtorus of $S$ defines an apartment containing $x\in\cB(G,\Qp)$.  Let $T:=\cZ_G(S)$ be the centraliser of $S$. Since $G$ is quasi-split over $\Qpbr$, it follows that $T$ is a maximal torus of $G$. 

\begin{rmksub}\label{rmk:Cartan}
Let $\cS^\circ$ and $\cT^\circ$ respectively denote the connected components of the N\'eron models of $S$ and $T$. Then by \cite[Proposition~4.6.4]{BruhatTits:RedGp2}, we have the following closed immersions
\[
\cS^\circ \hra \cT^\circ \cong \cZ_{\cG^\circ}(\cS^\circ) \hra \cG^\circ
\]
extending the natural maps on the generic fibres.
Therefore, it follows that we have
\[
T(\Qpbr)_1  = T(\Qpbr) \cap \breve\KK^\circ,
\]
and the inclusion $T(\Qpbr)_1\hra \breve\KK^\circ$ comes from a closed embedding of the parahoric group schemes $\cT^\circ\hra \cG^\circ$.
\end{rmksub}

\begin{defnsub}\label{def:AffWeyl}
Let $\N_S\subset G$ denote the normaliser of $S$ in $G$.
Recall that the extended affine Weyl group is defined to be
\[\wt W:=\N_S(\Qpbr)/T(\Qpbr)_1.\]
As $S$ and $T$ are defined over $\Qp$, the natural $\sig$-action on $G(\Qpbr)$ stabilises $\N_S(\Qpbr)$, which defines a $\sig$-action on $\wt W$. For $\tilde w\in \wt W$, we write $\dot{\tilde w}\in \N_S(\Qpbr)$ to denote a lift of $\tilde w$.

We define the following subgroup:
\[
\wt W_{\breve\KK^\circ}:=(\N_S(\Qpbr)\cap \breve\KK^\circ)/T(\Qpbr)_1 \subset \wt W.
\]
By \cite[Appendix, Proposition~12]{HainesRapoport:Gamma1p-Drinfeld}, $\wt W_{\breve\KK^\circ}$ coincides with the Weyl group of the maximal reductive quotient of $\cG^\circ_{\Fpbar}$ with respect to the maximal torus $\cS^\circ_{\Fpbar}$. In particular, $\wt W_{\breve\KK^\circ}$ is finite.
\end{defnsub}

\begin{rmksub}\label{rmk:AffWeyl}
We recall a few basic properties of $\wt W$ from  \cite[Appendix]{HainesRapoport:Gamma1p-Drinfeld}.
\begin{enumerate}
\item Let $W_0:=\N_S(\Qpbr)/T(\Qpbr)$ denote the Weyl group associated to the relative root system for $G_{\Qpbr}$. (Note that $W_0$ is the $I_{\Qp}$-invariant of the Weyl group of $G_{\Qpbar}$.) Then we have a following short exact sequence:
\[
\xymatrix@1{
 0 \ar[r]& X_*(T)_{I_{\Qp}} \ar[r]&\wt W \ar[r]& W_0 \ar[r]&1.
 }\]

If $x\in\cB(G,\Qp)$ is a special vertex, then $\wt W_{\breve\KK_x^\circ}$ maps isomorphically onto $W_0$ by the natural projection $\wt W\thra W_0$; \emph{cf.} \cite[Appendix, Proposition~13]{HainesRapoport:Gamma1p-Drinfeld}. In particular, we have $\wt W \cong X_*(T)_{I_{\Qp}}\rtimes W_0$, where the splitting depends on the choice of special vertex.
\item
There is a bijection
  \begin{equation}\label{eqn:CartanDecomp}
  \wt W_{\breve\KK^\circ}\backslash \wt W/\wt W_{\breve\KK^\circ} \riso \breve\KK^\circ\backslash G(\Qpbr)/\breve\KK^\circ
  \end{equation}
  sending $ \wt W_{\breve\KK^\circ}\tilde w\wt W_{\breve\KK^\circ}$ to $\breve\KK^\circ \dot{\tilde w}\breve\KK^\circ$; \emph{cf.}  \cite[Appendix, Proposition~8]{HainesRapoport:Gamma1p-Drinfeld}.

If $\breve\KK^\circ$ is an Iwahori subgroup (in which case    $ \wt W_{\breve\KK^\circ}$ is a trivial subgroup), the above bijection becomes $\wt W\riso \breve\KK^\circ\backslash G(\Qpbr)/\breve\KK^\circ$. If $\breve\KK^\circ$ is a special parahoric sugbroup (in which case we have $ \wt W_{\breve\KK^\circ}\cong W_0$), the above bijection becomes
\[
 X_*(T)_{I_{\Qp}}/W_0 \riso \breve\KK^\circ\backslash G(\Qpbr)/\breve\KK^\circ.
\]
In particular, if $\breve\KK^\circ$ is hyperspecial (so $X_*(T) = X_*(T)_{I_{\Qp}}$ and $W_0$ is the Weyl group for $G_{\Qpbar}$) then the left hand side is in bijection with the set of dominant cocharacters (with respect to some choice of borel subgroup), and the bijection recovers the Cartan decomposition.
\end{enumerate}
\end{rmksub}

\subsection{Review on $G$-(iso)crystals}\label{subsec:Gisoc}
\begin{defnsub}\label{def:grmu}
Let $\XX$ be an $\Zpbr$-scheme\footnote{It is often convenient and natural to allow $\XX$ to be an analytic space or a formal scheme. But it will be quite obvious how to adapt the subsequent discussion to these cases.}. For a cocharacter  $\mu:\Gm\ra \GL(M)_{\Zpbr}$, we say that a grading $\gr^\bullet(\OO_\XX\otimes_{\Zpbr} M)$ is \emph{induced from $\mu$}, if the $\Gm$-action on $\OO_\XX\otimes_{\Zpbr} M$ via $\mu$ leaves each grading stable and the resulting $\Gm$-action on $\gr^a(\OO_\XX\otimes_{\Zpbr} M)$ is given by
\[ \xymatrix@1{\Gm \ar[rr]^-{z\mapsto z^{-a}} & & \Gm \ar[rr]^-{z\mapsto z\id} & & \GL(\gr^a(\OO_\XX\otimes_{\Zpbr} M))}.\]
\end{defnsub}

Let $\breve{\Z}_p:=W(\Fpbar)$ and $\breve{\Q}_p:= W(\Fpbar)_\Q$, and we let $\sig$ denote the Witt vector Frobenius on $\breve{\Z}_p$ and $\breve{\Q}_p$.

\begin{defnsub}\label{def:nu}
Let $\bD$ be a pro-torus with character group $X^*(\bD) = \Q$; i.e., $\bD=\varprojlim\Gm$ where the transition maps is the $N$th power maps ordered by divisibility.
%
\end{defnsub}

We now work under the setting introduced in Definition~\ref{def:KP}; namely, let $G$ be a \emph{connected} reductive group over $\Qp$, and $\cG$ be a Bruhat-Tits integral model of $G$ as in Definition~\ref{def:KP}. We set $\breve\KK:=\cG(\Zpbr)$.
\begin{defnsub}
 %
For $b\in G(\breve{\Q}_p)$, we let $[b]:=\set{g\iv b\sigma(g),\ \forall g\in G(\breve{\Q}_p)}$ denote the $\sig$-$G(\breve{\Q}_p)$ conjugacy class of $b$. Similarly, we let $\db b:=\set{g\iv b \sigma(g);\ g\in \breve\KK}$ denote the $\sig$-$\breve\KK$ conjugacy class containing $b$.
\end{defnsub}

If $\cG = \GLn$, then the above definition has an interpretation in terms of $F$-(iso)crystals as follows. For $b,b'\in\GLn(\Qpbr)$, two $F$-isocrystals $(\Qpbr^n,b\sig)$ and $(\Qpbr^n, b'\sig)$ are  isomorphic if and only if $b$ and $b'$ are $\sig$-$\GLn(\Qpbr)$ conjugate. Similarly, two virtual $F$-crystals $(\Zpbr^n,b\sig)$ and $(\Zpbr^n, b'\sig)$ are  isomorphic if and only if $b$ and $b'$ are $\sig$-$\GLn(\Zpbr)$ conjugate. (By virtual $F$-crystal, we mean a $\Zpbr$-lattices in a $F$-isocrystal that is not necessarily $F$-stable.)

Kottwitz \cite[\S4]{Kottwitz:Gisoc1} showed that for $b\in G(\breve{\Q}_p)$ there exists a unique homomorphism
\[\nu_b:\bD\ra G_{\breve{\Q}_p}\]
such that for any representation $\rho:G_{\breve{\Q}_p}\ra \GL(n)_{\breve{\Q}_p}$ the $\Q$-grading associated to the inverse of $\rho\circ\nu_b$ is the slope decomposition for $(\breve{\Q}_p^n, \rho(b)\sig)$. (Note our sign convention in Definition~\ref{def:grmu}.) Furthermore, the uniqueness shows that for any $g,b\in G(\breve{\Q}_p)$ we have \[\nu_{gb\sigma(g)\iv} = g\nu_b g\iv.\]
Therefore, the $G(\breve{\Q}_p)$-conjugacy class of $\nu_b$ only depends on the $\sig$-$G(\breve{\Q}_p)$ conjugacy class $[b]$. So the $G(\breve{\Q}_p)$-conjugacy class of $\nu_b$, which will be denoted by $\nu_{[b]}$, only depends on the $\sig$-conjugacy class $[b]$.

The uniqueness of $\nu_b$ also implies that $\nu_{\sigma(b)} = \sigma^*\nu_b$. Since we have $\sigma(b) = b\iv b \sig(b)$, it follows that
\[
\sigma^*\nu_b = b\iv \nu_b b.
\]
In particular, the conjugacy class $\nu_{[b]}$ is $\sig$-stable; that is, the $G(\Qpbr)$-conjugacy class of cocharacters $\nu_{[b]}$ is defined over $\Qp$. (Note that unless $G$ is quasi-split over $\Qp$, this does not necessarily imply that $\nu_{[b]}$ contains a cocharacter defined over $\Qp$.)

\begin{defnsub}\label{def:Decency}
We say that $b\in G(\breve{\Q}_p)$ is \emph{decent} if for some $r\in\Z$ we have:
\[
(b\sig)^r = p^{r\nu_b}\sigma^r,
\]
where the equality takes place in $G(\breve{\Q}_p)\rtimes\langle\sig\rangle$. We call the above  equation a \emph{decency equation}. By \cite[Corollary~1.9]{RapoportZink:RZspace}, if $b$ is decent then we have $b\in G(\Q_{p^r})$, where $r$ is as in the decency equation.
\end{defnsub}

Kottwitz \cite[\S4]{Kottwitz:Gisoc1} showed that  any $\sig$-conjugacy class $[b]$ in $G(\breve{\Q}_p)$ contains a decent element provided that $G$ is a \emph{connected} reductive group over $\Qp$.

Consider the following group valued functor $J_b=J_{G,b}$ defined as follows:
\begin{equation}\label{eqn:J}
J_b(R):=\set{g\in G(R\otimes_{\Qp}\breve{\Q}_p)|\ gb\sigma(g)\iv = b}
\end{equation}
for any $\Qp$-algebra $R$. Note that for any $g,b\in G(\breve{\Q}_p)$ we have $J_{gb\sigma(g)\iv} (R) = g J_b(R) g\iv$ as a subgroup of $G(R\otimes_{\Qp}\breve{\Q}_p)$; in particular, $J_b$ essentially depends only on the $\sig$-conjugacy class of $b$ in $G(\breve{\Q}_p)$.

\begin{propsub}\label{prop:J}
The group valued functor $J_b$ can be represented by a connected reductive group over $\Qp$. Furthermore, there is a unique isomorphism
\[(J_b)_{\breve{\Q}_p}\riso G_{\nu_b}:=\cZ_G(\nu_b) \]
which induce that inclusion $J_b(\Qp) \hra  G_{\nu_b}(\breve{\Q}_p) \hra G(\breve{\Q}_p)$.
(Note that the centraliser $G_{\nu_b} \subset G_{\breve{\Q}_p}$ is a Levi subgroup.
\end{propsub}
\begin{proof}
We may assume that $b$ is decent, in which case the proposition was essentially proved in \cite[Corollary~1.14]{RapoportZink:RZspace}.
\end{proof}

The extended affine Weyl group $\wt W$ is a very useful tool to study $\sig$-$G(\Qpbr)$ and $\sig$-$\breve\KK$ conjugacy classes.
Let us first recall the following result of X.~He's on $\sig$-$G(\Qpbr)$ conjugacy classes  \cite[Theorem~3.7]{He:GeomADL}.
\begin{thmsub}\label{thm:He}
Using the notation from Definition~\ref{def:AffWeyl},
The natural map $\N_S(\Qpbr)\hra G(\Qpbr) \thra G(\Qpbr)/\sigma\dcj$ induces a surjective map
\[
\wt W/\sigma\dcj \thra G(\Qpbr)/\sigma\dcj.
\]
For $\tilde w\in\wt W$, we let $[\tilde w]$ denote the $\sigma$-conjugacy of $\tilde w$, which could also viewed as the $\sigma$-conjugacy class $[\dot{\tilde w}]$ of some lift $\dot{\tilde w}$ via the isomorphism above.
\end{thmsub}
Indeed, one can obtain from \cite{He:GeomADL} a much stronger result about the existence of combinatorially nice representative of $\wt W/\sigma\dcj$. 

\begin{lemsub}\label{lem:dbw}
Let  $\tilde w\in\wt W$, and choose two lifts $\dot{\tilde w},\dot{\tilde w}'\in \N_S(\Qpbr)$. Then there exists $u\in T(\Qpbr)_1$ such that $\dot{\tilde w}' = u \dot{\tilde w} \sigma(u)\iv$. In particular, for any stabiliser $\breve\KK$ of a facet $\Omega$ in the apartment corresponding to $S_{\Qpbr}$, the $\sig$-$\breve\KK$ conjugacy class $\db{\dot{\tilde w}}$ only depends on $\tilde w$, not on the choice of $\dot{\tilde w}$.
\end{lemsub}
\begin{proof}
We write $\tilde w' = t\tilde w$ for some $t\in T(\Qpbr)_1$, and we want to find an element $u\in T(\Qpbr)_1$ satisfying
\[
t = u\dot{\tilde w}\sigma(u)\iv\dot{\tilde w}\iv
\]
In other words, it suffices to show that the group homomorphism $\varphi_{\tilde w}:T(\Qpbr)_1\to T(\Qpbr)_1$ defined by $\varphi_{\tilde w}(u):= u\dot{\tilde w}\sigma(u)\iv\dot{\tilde w}\iv$ is surjective. 

Note that the conjugation by $\tilde w$ on $T(\Qpbr)_1$ only depends on the image $w\in W_0$ of $\tilde w$ via the natural projection $\wt W\thra W_0$.
Since $T(\Qpbr)_1 = \cT^\circ(\Zp)$ where $\cT^\circ$ is the connected N\'eron model of $T$, $\varphi_{\tilde w}$ turns out to be a Lang isogeny of $\cT^\circ$. Therefore, $\varphi_{\tilde w}$ is induces a surjective map on $\cT^\circ(\Zpbr) = T(\Qpbr)_1$ (as $\Zpbr$ is strictly henselian).

The second claim follows from the first since we have $T(\Qpbr)_1\subset \breve\KK^\circ\subset\breve\KK$ for any $\breve\KK$ as above. 
\end{proof}

\begin{defnsub}\label{def:dbw}
For $\tilde w\in\wt W$, let $\db{\tilde w}_{\breve\KK}$ denote the $\sig$-$\breve\KK$ conjugacy class of any lift $\dot{\tilde w}\in\N_S(\Qpbr)$. By Lemma~\ref{lem:dbw}, $\db{\tilde w}$ does not depend on the choice of $\dot{\tilde w}$.

If $\breve\KK$ is understood, then we write $\db{\tilde w}$ instead of $\db{\tilde w}_{\breve\KK}$.
\end{defnsub}

\begin{lemsub}\label{lem:Decency}
	For $\tilde w\in\wt W$, there exists a lift $\dot{\tilde w}$ that satisfies a ``decency equation''
	\[
(\dot{\tilde w}\sig)^r = p^{r\nu_{\dot{\tilde w}}}\sigma^r,
\]
for some $r\in\Z$; in other words, $\db{\tilde w}_{\breve\KK}$ contains a decent element.
\end{lemsub}
\begin{proof}
%
Since $W_0$ is finite and $\Gal(\Qpbr/\Qp)=\wh{\langle\sig\rangle}$ acts on $W_0$ through a finite quotient, there exists $r$ such that for any $w\in W_0$ we have $(w\sig)^r = \sigma^r$ in $W_0\rtimes \langle\sig\rangle$. Therefore for any $\dot{\tilde w}\in \N_S(\Qpbr)$ we have $(\dot{\tilde w}\sig)^r\in T(\Qpbr)\sigma^r$. We choose $r$ to also satisfy that $\cS^\circ$ splits over $\Z_{p^r}\subset\Zpbr$. Let $t\in T(\Qpbr)$ denote the element satisfying $t\sigma^r = (\dot{\tilde w}\sig)^r$.

Since $S_{\Qpbr}$ is the split part of $T_{\Qpbr}$, it follows that $X_*(S)$ is precisely the torsion-free part of $X_*(T)_{I_{\Qp}}$.
So by replacing $r$ with some suitable multiple, we may assume that the image of $t$ via $\kappa_T:T(\Qpbr)\thra X_*(T)_{I_{\Qp}}$ lies in $X_*(S)$. If we write $\lambda:=\kappa_T(t)\in X_*(S)$, then we may write $t=p^\lambda t_1$ for $t_1\in T(\Qpbr)_1$.

Now, by repeating the proof of Lemma~\ref{lem:dbw} we can find $u\in T(\Qpbr)_1$ such that $\sigma^r(u)u\iv = t_1\iv$.
For $u$ as above, we set $\dot{\tilde w}':=u\iv \dot{\tilde w} \sigma(u)$, which satisfies the following:
\[(\dot{\tilde w}'\sig)^r = (u\iv t \sigma^r(u)) \sigma^r = (t_1\iv t) \sigma^r = p^\lambda \sigma^r.\]
From this, it clearly follows that $\lambda = r\nu_{\dot{\tilde w}'}$; in particular, $\dot{\tilde w}'$ is decent.
\end{proof}

\subsection{$G$-(iso)crystals and virtual $F$-crystals with tensors}\label{subsec:VirCrys}
Let $R$ be either a field of characteristic zero or a discrete valuation ring of mixed characteristic. In practice, $R$ will be one of $\Q$, $\Z_{(p)}$, and $\Zp$.
Let $\cG$ be a smooth affine group scheme over $R$ such that the generic fibre is a reductive group. 
Let $M$ be a  free $R$-module of finite rank, and we fix a closed immersion of group schemes $\cG\hra \GL_R(M)$.

\begin{propsub}\label{prop:Chevalley}
In the above setting, there exists a finitely many elements $s_\alpha\in M^\otimes$ such that $\cG$ is the pointwise stabiliser of $(s_\alpha)$; i.e., for any $R$-algebra $R'$, we have
\[\cG(R') = \set{g\in \GL_R(M)(R'); \ g(s_\alpha) = s_\alpha \ \forall \alpha}.\]
\end{propsub}
\begin{proof}
The case when $R$ is a field is proved in \cite[Proposition~3.1]{Deligne:AbsHodge}, and the case of discrete valuation rings
is proved in \cite[Proposition~1.3.2]{Kisin:IntModelAbType}.
\end{proof}

We now return to the setting introduced in Definition~\ref{def:KP}.

Choose a finite free $\Zp$-module $\Lambda$ equipped with a faithful $\cG$-action (i.e., a closed embedding $\cG \hra \GL(\Lambda)$ of algebraic groups over $\Zp$). Fixing such a datum, we can choose finitely many tensors $(s_\alpha)\subset \Lambda^\otimes$ defining $\cG$ as a subgroup of $\GL(\Lambda)$ by Proposition~\ref{prop:Chevalley}.

For any $b\in G(\Qpbr) = \cG(\Qpbr)$, we obtain the following virtual $F$-crystal over $\Fpbar$
\begin{equation}\label{eqn:bM}
\bM_b (=\bM_b^\Lambda):= (\Zpbr\otimes_{\Zp}\Lambda, b\sig).
\end{equation}
(Here, we identify $\cG$ as a subgroup of $\GL(\Lambda)$.)
In the intended setting, $\bM_b$ will be assumed to be the dual of the contravariant Dieudonn\'e module of a $p$-divisible group $\BX_b$ over $\Fpbar$ and $\Lambda$ will be isomorphic to the Tate module of a suitable $\Zpbr$-lift of $\BX_b$. (In other words, one can identify $\bM_b$ as the underlying $\Zpbr$-module for the covariant Dieudonn\'e module of $\BX_b$, and $b\sigma = p\iv F$ where $F$ is the covariant crystalline Frobenius operator. In particular, $\bM_b$ is not stable under $b\sig$ unless it is \'etale. The reason for this normalisation is to identify $\bM_b$ as the first crystalline \emph{homology} without Tate twist.)

Let $\triv:=(\Zpbr,\sig)$ denote the trivial (virtual) $F$-crystal.
For each $s_\alpha$, let us consider the following $\Zpbr$-linear map
\begin{equation}\label{eqn:bfs}
\triv\ra\bM_b^\otimes = \Zpbr\otimes_{\Zp}\Lambda^\otimes;\ 1\mapsto 1\otimes s_\alpha.
\end{equation}
Since (the image of) $\cG$ in $\GL(\Lambda)$ is the pointwise stabiliser of $(s_\alpha)$, it follows that (\ref{eqn:bfs}) is a map of virtual $F$-crystals for each $\alpha$ (i.e., a $\Zpbr$-linear morphism which induces a morphism of $F$-isocrystals after inverting $p$).

Let $b'\in G(\Qpbr)$. Then clearly, we have $b'\in [b]$ if and only if $[\bM_b[\ivtd p],(s_\alpha)]\cong [\bM_{b'}[\ivtd p],(s_\alpha)]$ (i.e., there exists an isomorphism $\bM_b[\ivtd p]\cong \bM_{b'}[\ivtd p]$ of $F$-isocrystals which fixes each $ s _\alpha$). Similarly, $b'$ lies in the $\sig$-$\breve\KK$ conjugacy class $\db b$  if and only if we have $[\bM_b,(s_\alpha)]\cong [\bM_{b'},(s_\alpha)]$.
%
%
%

\subsection{Completely slope divisible $G$-crystals}
The following definition is an analogue of completely slope divisible $p$-divisible groups over $\Fpbar$. We will not directly work with this definition, but we will use it as a motivation for regarding $\db{\tilde w}$ as a ``completely slope divisible virtual $F$-crystals with tensors''; \emph{cf.} Lemma~\ref{lem:AffWeyl}.
\begin{defnsub}\label{def:SlopeDiv}
An element $b\in \cG(\breve{\Q}_p)$  is called \emph{completely slope divisible} if  $b$ is decent and $\nu_{b}:\bD\ra  G_{\breve{\Q}_p}$ is the base change of a homomorphism $\bD\ra\cG_{\Zpbr}$ over $\Zpbr$. A $\sig$-$\breve\KK$ conjugacy class $\db b$ is called \emph{completely slope divisible} if it contains a completely slope divisible element.
\end{defnsub}

The motivation of Definition~\ref{def:SlopeDiv} is the following (more standard) definition of completely slope divisible $p$-divisible groups:
\begin{defnsub}\label{def:SlopeDivBT}
Let $\BX$ be a $p$-divisible group over $\Fpbar$ with height $n$.  We say that $\BX$ is \emph{completely slope divisible} if it admits a  filtration
\[
0= \BX_0 \subsetneq \BX_1 \subsetneq \cdots\subsetneq \BX_m = \BX
\]
such that for some integer $r>0$ and strictly decreasing sequence of  integers $a_i\in[0,r]$ the quasi-isogeny $p^{-a_i}\Fr^{r}:\BX_i \ra (\BX_i)\com{p^r}$ is an isogeny inducing an isomorphism $\BX_i/\BX_{i-1} \riso (\BX_i/\BX_{i-1})\com{p^r}$. We call $\set{\lambda_i:=a_i/r}$ the slopes of $\BX$.

We write $\BX\com i:= \BX_i/\BX_{i-1}$. Since the slope filtration has a canonical splitting over a perfect base, we may write $\BX = \prod_i\BX\com i$;  \emph{cf.} \cite[Corollary~11]{Zink:SlopeFil}.
\end{defnsub}
Recall that the isomorphism class of $\BX$ corresponds to certain $\sig$-$\GL_n(\Zpbr)$ conjugacy class $\db b$.
One can see that $\BX$ is completely slope divisible if and only if $\db b$ is completely slope divisible. Indeed, if $\BX$ is completely slope divisible, then we choose a $\Zpbr$-basis of $\DD(\BX)(\Zpbr)^*$ adapted to the slope decomposition $\BX = \prod_i\BX\com i$, and require that the basis of $\DD(\BX^{(i)})(\Zpbr)^*$ is fixed by $p^{-a_i}\Fr^r$. Then, the resulting matrix $b\in \GLn(\Qpbr)$ is clearly decent and $\nu_b$ is defined over $\Zpbr$ (since the slope decomposition is defined over $\Zpbr$). Conversely, if $b\in\GL_n(\Qpbr)$ is completely slope divisible, then the slope decomposition  of the $F$-isocrystal $\bM_b[\ivtd p]$ restricts to the slope decomposition of $\bM_b$, which allow us to write $\BX_b = \prod \BX_b^{(i)}$. Now, the decency equation for $b$ implies that there exists $r$ and $a_i$ such that the quasi-isogeny $p^{-a_i}\Fr^r:\BX_b^{(i)} \to (\BX_b^{(i)})^{(p^r)}$ is an isomorphism.

Now let $\db b$ denote a $\sig$-$\breve\KK$ conjugacy class of $b\in G(\Qpbr)$, and we choose $\cG\hra \GL_n$ such that $\db b$ corresponds to an isomorphism class of $p$-divisible groups of height~$n$ with tensors $(\BX, (s_\alpha))$. Then if $\db b$ is completely slope divisible, then $\BX$ is completely slope divisible, and the tensors $(s_\alpha)$ are  ``well behaved'' with respect to the slope decomposition.

\begin{lemsub}\label{lem:AffWeyl}
We use the notation from Definition~\ref{def:AffWeyl}.
For any $\tilde w\in\wt W$, any decent lift $\dot{\tilde w}$ of $\tilde w$ is completely slope divisible in the sense of Definition~\ref{def:SlopeDiv}. In particular,
the $\sig$-$\breve\KK$ conjugacy class $\db{\tilde w}$ (as in Definition~\ref{def:dbw}) is completely slope divisible. 
\end{lemsub}

\begin{proof}
Let us choose a \emph{decent} lift $\dot{\tilde w}$, which exists by Lemma~\ref{lem:Decency}. The proof of Lemma~\ref{lem:Decency} also shows that $\nu_{\dot{\tilde w}}:\bD \ra G_{\Qpbr}$ factors through $S_{\Qpbr}$, so it extends to a rational cocharacter $\bD \ra \cS^\circ_{\Zpbr}$ over $\Zpbr$.\footnote{By replacing $\nu_{\dot{\tilde w}}$ with $r\nu_{\dot{\tilde w}}$ for some $r$, the claim reduces to a familiar statement on  cocharacters of split tori.} Since $\cS^\circ$ is a torus in $\cG^\circ$, it follows that $\nu_{\dot{\tilde w}}$ is defined over $\Zpbr$. Therefore, $\dot{\tilde w}$ is completely slope divisible, and so the same holds for its $\sigma$-$\breve\KK$ conjugacy class $\db{\tilde w}$.
\end{proof}


Recall that we have a natural bijection $\wt W_{\breve\KK^\circ}\backslash \wt W/\wt W_{\breve\KK^\circ} \cong \breve\KK^\circ\backslash G(\Qpbr)/\breve\KK^\circ$ given by $\wt W_{\breve\KK^\circ}\tilde w\wt W_{\breve\KK^\circ}\mapsto \breve\KK^\circ\dot{\tilde w}\breve\KK^\circ$, where $\dot{\tilde w}$ is any lift of $\tilde w$; \emph{cf.} Remark~\ref{rmk:AffWeyl}. For a dominant cocharacter $\nu\in X_*(T)_+$ defined over $\ol{\breve\Q}_p$ (where dominance is with respect to the choice of a suitable borel subgroup), one can associate a subset $\Adm_{\breve\KK^\circ}(\nu)\subset \wt W_{\breve\KK^\circ}\backslash \wt W/\wt W_{\breve\KK^\circ}$; \emph{cf.} \cite[(3.6)]{Rapoport:GuideRednShVar}. 

\begin{defnsub}
For $b\in G(\breve{\Q}_p)$ and $\nu\in X_*(T)_+$, we consider an \emph{affine Deligne-Lusztig variety of level $\breve\KK^\circ$}, defined as follows
\[
X_{\breve\KK^\circ}(b;\nu):=\bigcup_{\tilde w \in \mathrm{Adm}_{\breve\KK^\circ}(\nu)}\set{g\in G(\breve{\Q}_p)|\ g\iv b \sigma(g) \in \breve\KK^\circ\dot{\tilde w} \breve\KK^\circ}/\breve\KK^\circ \subset G(\breve{\Q}_p)/\breve\KK^\circ.
\]
And we let $X_{\breve\KK}(b;\nu)\subset G(\Qpbr)/\breve\KK$ denote the image of $X_{\breve\KK^\circ}(b;\nu)$ under the natural projection, and call it an \emph{affine Deligne-Lusztig variety of level $\breve\KK$}. 

Note that for any $g\in G(\breve{\Q}_p)$, the $\sig$-$\breve\KK^\circ$ conjugacy class  $\db{g\iv b\sigma(g)}_{\breve\KK^\circ}$ only depends on the right coset $g\breve\KK^\circ$, and the same assertion holds if we replace $\breve\KK^\circ$ with $\breve\KK$. We say that $g\breve\KK\in X_{\breve\KK}(b;\nu)$ is a \emph{completely slope divisible point} if $\db{g\iv b\sigma(g)}_{\breve\KK}$ is completely slope divisible.
\end{defnsub}

The following proposition can be thought of as the slope filtration theorem for  affine Deligne-Lusztig varieties. 
\begin{propsub}\label{prop:SlopeFil}
For $b\in G(\breve{\Q}_p)$ and $\nu\in X_*(T)_+$, assume that $X_{\breve\KK^\circ}(b;\nu)$ is non-empty. Then 
there exists $g\breve{\KK}^\circ \in X_{\breve\KK^\circ}(b;\nu)$ such that $\db{g\iv b \sigma(g)}_{\breve{\KK}^\circ} =  \db{\tilde w}_{\breve{\KK}^\circ}$ for some $\tilde w\in\widetilde W$. In particular, $X_{\breve\KK^\circ}(b;\nu)$ admits a completely slope divisible point. 

The same statement holds for $X_{\breve\KK}(b;\nu)$.
\end{propsub}
Note that for $\nu'\leq \nu$, we have $\Adm_{\breve\KK^\circ}(\nu')\subseteq \Adm_{\breve\KK^\circ}(\nu)$, so $X_{\breve\KK}(b;\nu') \subseteq X_{\breve\KK}(b;\nu)$. Later, we will only consider the case when $\nu$ is minuscule.

\begin{proof}
It suffices to prove the assertion for $X_{\breve\KK^\circ}(b;\nu)$, which follows from  \cite[Corollary~2.9]{He:KottwitzRapoportConj}. Indeed, assuming that $X_{\breve\KK^\circ}(b;\nu)$ is non-empty, the aforementioned result implies that $[b]$ contains a lift $\dot{\tilde w}\in\N_S(\Qpbr)$ of $\tilde w\in\wt W$ such that $\wt W_{\breve\KK^\circ}$-double coset of $\tilde w$ belongs to $\mathrm{Adm}_{\breve\KK^\circ}(\nu)$. Therefore, there exists $g\breve{\KK}^\circ \in X_{\breve\KK^\circ}(b;\nu)$ such that $\db{g\iv b \sigma(g)}_{\breve{\KK}^\circ} =  \db{\tilde w}_{\breve{\KK}^\circ}$, which is completely slope divisible by Lemma~\ref{lem:AffWeyl}.
\end{proof}

\subsection{Affine Deligne-Lusztig varieties and $p$-divisible groups}
In this subsection, we want to show that under some suitable assumptions we can interpret $X_{\breve\KK}(b;\nu)$ as the set of self quasi-isogenies of some $p$-divisible group over $\Fpbar$ (\emph{cf.} Remark~\ref{rmk:EmbADLV}). We begin with the example of $\GLn$.
\begin{exasub}\label{exa:SlopeFil}
Let $G:=\GL_n$ and $ \breve\KK^\circ := \GL_n(\Zpbr)(=\breve\KK)$. Choose $b\in\GL_n(\Qpbr)$ so that we have a $p$-divisible group $\BX_b$ with $\big(\DD(\BX_b)(\Zpbr)\big)^* \cong \bM_b:=(\Zpbr^n,b\sigma)$ as a virtual $F$-crystal. Then it turns out that $[b]$ is neutral acceptable for a minuscule cocharacter $\nu_d$ given by $t \mapsto \diag(\underbrace{1,\cdots,1}_{n-d},\underbrace{ t\iv,\cdots,t\iv}_{d})$, where $d$ is the dimension of $\BX_b$. (This is a group-theoretic interpretation of Mazur's inequality; \emph{cf.}  \cite[Theorem~4.2]{RapoportRichartz:Gisoc}.) Now we can interpret $X_{\GL_n(\Zpbr)}(b;\nu_d)$ as the set of equivalence classes of pairs $(\BY, \iota:\BY\to\BX_b)$ where $\BY$ is a $p$-divisible group over $\Fpbar$ and $\iota$ is a quasi-isogeny. To describe this bijection, for any $g\in G(\Qpbr)$ the virtual $F$-crystal $\bM_{g\iv b\sigma(g)}:=(\Zpbr^n,g\iv b\sigma(g)\sigma)$ 
corresponds to a $p$-divisible group $\BY:=\BX_{g\iv b\sigma(g)}$ if and only if the right $\GLn(\Zpbr)$-coset of $g$ lies in $X_{\GL_n(\Zpbr)}(b;\nu_d)$ (which can be seen by classical Dieudonn\'e theory and \cite[Example~4.3]{RapoportRichartz:Gisoc}). In this case, the isomorphism of $F$-isocrystals $g:(\Qpbr^n,g\iv b\sigma(g)\sigma)\riso (\Qpbr^n,b\sigma)$ gives rise to a quasi-isogeny $\iota:\BY\to\BX_b$, which only depends on $g\GLn(\Zpbr)$.
\end{exasub}
Now let us make the following assumptions on $(\cG,b,\nu)$ and the \emph{closed immersion} $\rho:\cG\hra \GL(\Lambda)\cong\GL_n$ of smooth affine group schemes over $\Zp$.
\begin{equation}\label{eqn:RunningAssumption1}
G \text{ is split after a tame extension of }\Qp
\text{ and } p\nmid |\pi_1(G^{\der})|.
\end{equation}

Let us briefly recall the theory of local models in this setting, following Pappas and Zhu \cite{PappasZhu:LocMod}.
Under (\ref{eqn:RunningAssumption1}), Pappas and Zhu \cite[Theorem~4.1]{PappasZhu:LocMod} constructed a smooth affine group scheme $\underline\cH$ over $\Zp[u]$ with connected fibres with the following properties
\begin{enumerate}
\item $\cG^\circ \cong \underline\cH\times_{\Spec\Zp[u]}\Spec\Zp[u]/(u-p)$
\item $\underline{\cH}\times_{\Spec\Zp[u]}\Spec\Zp[u^{\pm}]$ is a reductive group scheme and admits a rigidification in the sense of \cite[\S3.3.5]{PappasZhu:LocMod} (which involves a choice of maximal split torus over $\Zp[u^{\pm}]$, maximal split torus over $\Zpbr[u^{\pm}]$ defined over $\Zp[u^{\pm}]$, etc.). In particular, if we set $H:=\underline\cH\times_{\Spec\Zp[u]}\Spec\Fp\llpar u\rrpar$ the choice of rigidification gives rise to an identification of the Bruhat-Tits buildings $\cB(G,\Qp)\cong \cB(H,\Fp\llpar u \rrpar)$ and $\cB(G,\Qp^{\ur})\cong \cB(H,\Fp\llpar u \rrpar^{\ur})$, identifying the Iwahori Weyl groups. 
\item Recall that we chose $x\in\cB(G,\Qp)$ so that $\cG^\circ = \cG^\circ_x$. Viewing $x$ as a point in $\cB(H,\Fp\llpar u\rrpar)$ by the above identification,  $\cH:=\underline{\cH}\times_{\Spec\Zp[u]}\Spec\Fp[[u]]$ is the parahoric group scheme associated to $x$.
\end{enumerate}

Using the group scheme $\underline{\cH}$ extending $\cG^\circ$, Pappas and Zhu \cite[Definition~7.1]{PappasZhu:LocMod} constructed a local model $\mathrm{M}^{\loc}_{\cG^\circ,\{\nu\}}$ over $\fo_E$ where $E$ is the field of definition of $\{\nu\}$, so that $\mathrm{M}^{\loc}_{\cG^\circ,\{\nu\}}$ has a natural action of $\cG^\circ$ and its geometric special fibre has a description in terms of affine Schubert variety of $\cH$ (\emph{cf.} \cite[Theorem~9.3]{PappasZhu:LocMod}). This construction applies to $(\cG^\circ,\{\nu\}) = (\GLn,\{\nu_d\})$, and the resulting local model $\mathrm{M}^{\loc}_{\GLn,\{\nu_d\}}$ is the grassmannian scheme over $\Zp$, classifying rank-$d$ quotients of a fixed rank-$n$ trivial vector bundle.

Now let us make the following additional assumption:
\begin{equation}\label{eqn:RunningAssumption2}
 \rho\circ\nu \text{ is }\GL_n(\overline{\breve\Q}_p)\text{-conjugate to }\nu_d\text{ for some } d.
\end{equation}
Then Kisin and Pappas showed that the locally closed immersion $\rho|_{\cG^\circ}:\cG^\circ\hra \GLn$ extends to a locally closed immersion 
\begin{equation}\label{eqn:rho-affine}
\underline\rho:\underline\cH\to\GLn
\end{equation}
of smooth affine group schemes over $\Zp[u]$; \emph{cf.} \cite[proof of Proposition~2.3.7]{KisinPappas:ParahoricIntModel}.
Furthermore, by \cite[Theorem~8.1]{PappasZhu:LocMod} it follows that  $\underline{\rho}$ induces a \emph{closed immersion} 

\begin{equation}\label{eqn:EmbLocMod}
\mathrm{M}^{\loc}_{\cG^\circ,\{\nu\}} \hra \mathrm{M}^{\loc}_{\GLn,\{\nu_d\}}\times_{\Spec\Zp}\Spec\fo_E.
\end{equation}

The following result is essentially a consequence of the theory of local models (\emph{cf.} \cite[Corollary~3.5]{Zhou:ModpPts}).
\begin{propsub}\label{prop:EmbADLV}
Assume that conditions (\ref{eqn:RunningAssumption1}, \ref{eqn:RunningAssumption2}) hold for $(\cG,b,\nu)$ and $\rho:\cG\hra \GL_n$.
Then for any $\nu$-admissible element $\tilde w\in \widetilde W$ and any lift $\dot{\tilde w}\in\cG^\circ(\Zpbr)$, we have $\rho(\dot{\tilde w})\in \GLn(\Zpbr)p^{\nu_d}\GLn(\Zpbr)$ and so 
$\rho$ induces an injective map
\[
X_{\breve\KK}(b;\nu) \hra X_{\GL_n(\Zpbr)}(\rho(b);\nu_d) .
\]
\end{propsub}
\begin{proof}
The first claim is proved \cite[Corollary~3.5]{Zhou:ModpPts}. Note that its proof does not use the running assumption in \cite[\S3.2]{Zhou:ModpPts} that $\underline\rho$ (\ref{eqn:rho-affine}) is a closed immersion (i.e., $\cG^\circ=\cG$). Indeed, even when $\underline\rho$ is a locally closed immersion, we still have the closed immersion of the local models (\ref{eqn:EmbLocMod}) so we can inject $\mathrm{M}_{\cG^\circ,\set\mu}^{\loc}(\Fpbar)$ simultaneously into $\GLn(\Qpbr)/\GLn(\Zpbr)$ and $\GLn(\Fpbar\llpar u\rrpar)/\GLn(\Fpbar[[u]])$, respecting the identification of the embeddings of the Bruhat-Tits buildings $\cB(G,\Qpbr)\hra \cB(\GLn,\Qpbr)$ and $\cB(H,\Fpbar\llpar u\rrpar)\hra \cB(\GLn,\Fpbar\llpar u\rrpar)$. Then the rest of the proof of \cite[Proposition~3.4]{Zhou:ModpPts} goes through, using the description of the special fibre $\overline{\mathrm{M}}^{\loc}_{\cG^\circ,\set\mu}$ given in \cite[Theorem~9.3]{PappasZhu:LocMod}).

Now the map $\rho:G(\Qpbr)/\breve\KK^\circ \to \GL_n(\Qpbr)/\GL_n(\Zpbr)$ restricts to $X_{\breve\KK^\circ}(b;\nu)\to X_{\GLn(\Zpbr)}(\rho(b);\nu_d)$. Now since $\rho$ induces 
\[G(\Qpbr)/\breve\KK^\circ \thra G(\Qpbr)/\breve\KK  \hra \GL_n(\Qpbr)/\GL_n(\Zpbr)\] 
 and $X_{\breve\KK}(b;\nu)\subset G(\Qpbr)/\breve\KK$ is the image of  $X_{\breve\KK^\circ}(b;\nu)$, we obtain the desired injective map $X_{\breve\KK}(b;\nu) \hra X_{\GL_n(\Zpbr)}(\rho(b);\nu_d)$ as claimed.
\end{proof}

\begin{rmksub}\label{rmk:EmbADLV}
If one were to be optimistic, one can expect that Proposition~\ref{prop:EmbADLV} should hold even if we remove (or weaken) 
 assumption (\ref{eqn:RunningAssumption1}).
Indeed, what is really needed in the proof of Proposition~\ref{prop:EmbADLV} is a good theory of local models, which was  developed by Pappas and Zhu under assumption (\ref{eqn:RunningAssumption1}); \emph{cf.} \cite{PappasZhu:LocMod}.

If $\rho$ induces a natural map
$X_{\breve\KK}(b;\nu)\hra X_{\GLn(\Zpbr)}(\rho(b);\nu_d)$, then by Example~\ref{exa:SlopeFil} we get a $p$-divisible group $\BX_b(=\BX_{\rho(b)})$ with $(\DD(\BX_b)(\Zpbr))^*\cong (\Zpbr^n,\rho(b)\sigma)$. Furthermore,  
for any $g\in G(\Qpbr)$ whose right $\breve\KK$-coset lies in $X_{\breve\KK}(b;\nu)$, the isomorphism $\rho(g):(\Qpbr^n,\rho(g)\iv \rho(b)\sigma(\rho(g))\sigma)\riso (\Qpbr^n,\rho(b)\sigma)$ preserves the tensors $(s_\alpha)$, so the quasi-isogeny $\iota:\BY\to\BX_b$  corresponding to $\rho(g)\GLn(\Zpbr)\in X_{\GLn(\Zpbr)}(\rho(b);\nu_d)$ is  ``tensor-preserving''. In particular, Proposition~\ref{prop:SlopeFil} shows the existence of a ``tensor-preserving'' quasi-isogeny $\iota:\BY\to\BX_b$ where $\BY$ is completely slope divisible.
\end{rmksub}

\section{The group of tensor-preserving self quasi-isogenies}\label{sec:Qisg}

\subsection{Tensor-preserving internal hom $p$-divisible groups}
Inspired by \cite[\S4.1]{CaraianiScholze:ShVar}, we construct a $p$-divisible group $\cH^G_b$ over $\Fpbar$ that can be thought of as a ``tensor-preserving'' internal hom $p$-divisible group of $\BX_b$ (for completely slope divisible $b$).

Let $\BX$ and $\BX'$ be completely slope divisible $p$-divisible groups over $\Fpbar$. In particular, we have $\BX = \prod_i\BX^{(i)}$ and $\BX' = \prod_j\BX'^{(j)}$, where $\BX^{(i)}$ and $\BX'^{(j)}$ are pure of slope $\lambda_i$ and $\lambda'_j$, respectively. (\emph{Cf.} Definition~\ref{def:SlopeDivBT}.)
Then one can construct the ``internal hom $p$-divisible group'' $\cH_{\BX,\BX'}$ with the following properties:
\begin{enumerate}
\item We have an isomorphism of sheaves $\varprojlim \cH_{\BX,\BX'}[p^n] \cong \lhom(\BX,\BX')$ of $\Zp$-modules over $\Fpbar$.
\item Let $\bM:=\big(\DD(\BX)(\Zpbr)\big)^*$, $\bM':=\big(\DD(\BX')(\Zpbr)\big)^*$, and $\bM_\cH:= \big(\DD(\cH_{\BX,\BX'})(\Zpbr)\big)^*$ be the dual of the contravariant Dieudonn\'e modules (as virtual $F$-crystals). Then $\bM_\cH[\ivtd p]$ is the non-positive slope part of $\Hom(\bM, \bM')[\ivtd p]$.
\end{enumerate}
Indeed, when both $\BX$ and $\BX'$ are pure of some slopes, then $\cH_{\BX,\BX'}$ can be constructed as in \cite[\S4.1]{CaraianiScholze:ShVar}. For completely slope divisible $p$-divisible groups, we put 
\[
\cH_{\BX,\BX'}:=\prod_{i,j}\cH_{\BX^{(i)},\BX'^{(j)}}.
\]
This clearly satisfies the first property, and the second property follows from  \cite[Lemma~4.1.10]{CaraianiScholze:ShVar}. 

For any $p$-divisible group $\cH$ over $\Fpbar$, we define $\wt \cH:=\varprojlim_p\cH_{\Zp}$ where $\cH_{\Zp}$ is any $\Zp$-lift of $\cH$ and the limit is as an \emph{fpqc} sheaf on $\Nilp_{\Zpbr}^{op}$. By rigidity of quasi-isogeny, $\wt\cH$ is independent of choice of $\cH_{\Zp}$, and furthermore,  we have we have a natural isomorphism  $\wt\cH(R) \riso \wt\cH(R/p)$ for any $R\in\Nilp_{\Zpbr}$; \emph{cf.} \cite[Proposition~3.1.3(ii)]{ScholzeWeinstein:RZ}. 

Let us focus on $\cH_{\BX,\BX}$ with $\BX'=\BX$, where  $\BX$ is completely slope divisible. Then given any quasi-isogenies $\iota:\BX\to\BY$, we have  for any $R\in\Nilp_{\Zpbr}$
\begin{equation}\label{eqn:QEnd}
\wt\cH_{\BX,\BX}(R) := \End_{R/p}(\BX_{R/p})[\ivtd p]  \xrightarrow[(*)]{\sim} \End_{R/p}(\BY_{R/p})[\ivtd p]\xleftarrow[(**)]{\sim} \End(\BY_R)[\ivtd p] ,
\end{equation}
where $\BY_R$ is any lift of $\BY_{R/p}$, $(**)$ is  from \cite[Lemma~1.1.3]{Katz:SerreTate}, and $(*)$ is defined by sending $f:\BX_{R/p}\to\BX_{R/p}$ to $\iota_{R/p}\circ f\circ\iota_{R/p}\iv$.


Let $R$ be an f-semiperfect $\Fpbar$-algebra in the sense of \cite[Definition~4.1.2]{ScholzeWeinstein:RZ}. Let $\Acris(R)$ denote the universal $p$-adic PD thickening of $R$  (\emph{cf.} \cite[Proposition~4.1.3]{ScholzeWeinstein:RZ}), and set $\Bcris^+(R):=\Acris(R)[\ivtd p]$. Then for any $p$-divisible group $\BX$ over $\Fpbar$, we have a natural isomorphism $(\DD(\BX_R)(\Acris(R)))^* \cong \Acris(R)\otimes_{\Zpbr}(\DD(\BX)(\Zpbr))^*$ of virtual  $\Acris$-crystals.
\begin{defnsub}\label{def:TenHom}
Let $\BX$ be a $p$-divisible group over $\Fpbar$ with $\bM:=\big(\DD(\BX)(\Zpbr)\big)^*$. We fix finitely many $F$-invariant tensors $s_\alpha\in\bM^\otimes[\ivtd p]$; e.g. (\ref{eqn:bfs}). Then for any f-semiperfect $\Fpbar$-algebra $R$, we say that $\gamma \in \End(\BX_{R})[\ivtd p]$ is a \emph{tensor-preserving quasi-endomorphism} if the induced endomorphism of $\Bcris^+(R)\otimes_{\Zpbr}\bM$ preserves the tensors $(s_\alpha)\subset \Bcris^+(R)\otimes_{\Zpbr}\bM^\otimes$. We let $\End_{(s_\alpha)}(\BX_R)[\ivtd p]\subset \End((\BX_R)[\ivtd p]$ denote the subset of tensor-preserving quasi-endomorphisms.
\end{defnsub}

From now on, we choose $(b,\nu)$ and $\rho:\cG\hra\GL(\Lambda)\cong\GLn$, and assume that the following hold
\begin{description}
\item[{$(\dagger)$}] We have $X_{\breve{\KK}}(b;\nu)\ne\emptyset$ and $\rho$ induces $X_{\breve\KK}(b;\nu)\to X_{\GLn(\Zpbr)}(\rho(b);\nu_d)$ for some $d$.
\end{description}
(By Proposition~\ref{prop:EmbADLV} and \cite[Theorem~A]{He:KottwitzRapoportConj}, this condition is satisfied if (\ref{eqn:RunningAssumption1}, \ref{eqn:RunningAssumption2}) hold and $[b]$ is $\nu$-admissible.) Then by modifying $b$ up to $\sigma$-$G(\Qpbr)$ conjugacy so that the identity right coset of $\breve\KK$ lies in $X_{\breve\KK}(b;\nu)$, we can get a $p$-divisible group $\BX_b$ with $\bM_b\cong (\DD(\BX_b)(\Zpbr))^*$ as a virtual isocrystal (\emph{cf.} Example~\ref{exa:SlopeFil}), and $F$-invariant tensors $(s_\alpha)\subset \bM_b^\otimes$ as in \S\ref{subsec:VirCrys}.

In this setting, we can represent the  \emph{fpqc} sheaf of $\Qp$-vector spaces $R\rightsquigarrow \End_{(s_\alpha)}((\BX_b)_R)[\ivtd p]$  on the (opposite) category of f-semiperfect $\Fpbar$-algebras as follows. 
\begin{lemsub}\label{lem:TenHom}
In the above setting, there exists a $p$-divisible group $\cH^G_b$ (well defined up to isogeny) such that for any f-semiperfect $\Fpbar$-algebra $R$ we have a natural $\Qp$-linear isomorphism
\[
\wt\cH^G_b(R) \cong \End_{(s_\alpha)}((\BX_b)_R)[\ivtd p].
\]
We can obtain $\cH^G_b$ as a $p$-divisible subgroup of $\cH_b$ such that for any f-semiperfect $\Fpbar$-algebra $R$ we have $\wt\cH_b(R) \cong \End_R((\BX_b)_R)[\ivtd p]$.
\end{lemsub}

As discussed above, $\wt\cH^G_b$ can be canonically lifted to a formal group scheme over $\Spf\Zpbr$ so that $\wt \cH^G_b(R) = \wt \cH^G_b(R/p)$ for any $R\in\Nilp_{\Zpbr}$. So the above lemma also gives a description of points valued in $R\in\Nilp_{\Zpbr}$ when $R/p$ is f-semiperfect.
\begin{proof}
By Proposition~\ref{prop:SlopeFil}, there exists a completely slope divisible point $g\breve\KK \in X_{\breve\KK}(b;\nu)$ such that $\db{g\iv b\sigma(g)} = \db{\tilde w}$ for some $\tilde w\in\wt W$. We choose a decent lift $\dot{\tilde w}$ of $\tilde w$ (\emph{cf.} Lemma~\ref{lem:Decency}) and a coset representative $g$ of $g\breve{\KK}$ with $\dot{\tilde w} = g\iv b\sigma(g)$. Then by setting $\bM':=(\Zpbr^n,\rho(\dot{\tilde w})\sigma)$, we have a $p$-divisible group $\BY$ such that $(\DD(\BY)(\Zpbr))^*\cong\bM'$ as a virtual $F$-crystal. Furthermore, 
$g\in\GLn(\Qpbr)$ induces a tensor-preserving isomorphism of $F$-isocrystals $bM'[\ivtd p],(s_\alpha),(s_\alpha))\riso (\bM_{b}[\ivtd p],(s_\alpha))$, which induces a tensor-preserving quasi-isogeny $\iota:\BY\to\BX_b$.

We set $\cH_b:=\cH_{\BY,\BY}$, which makes sense as $\BX_{\dot{\tilde w}}$ is completely slope divisible (\emph{cf.} Lemma~\ref{lem:AffWeyl}). Then (\ref{eqn:QEnd}) gives rise to an  isomorphism $\wt\cH_b(R) \cong \End_R((\BX_b)_{R})[\ivtd p]$ (only depending on $\iota:\BY\to\BX_b$).

We set $\bM':=(\DD(\BY)(\Zpbr))^*$. Then since  $\End_{(s_\alpha)}(\bM')[\ivtd p]\subset \End(\bM')[\ivtd p]$ is a sub-$F$-crystal, and it has a $\Zp$-lattice given by tensor-preserving endomorphisms  $\End_{(s_\alpha)}(\bM')$. 
We define $\cH^G_{b}$ to be the $p$-divisible subgroup of $\cH_{b}:=\cH_{\BY,\BY}$ such that its dual Dieudonn\'e module $\bM^G_\cH$ satisfies
\[
(\bM^G_\cH) = \bM_\cH \cap \End_{(s_\alpha)}(\bM')[\ivtd p]\subset \End(\bM')[\ivtd p].
\]
By construction, $\bM^G_\cH[\ivtd p]$ is the non-positive slope part of the $F$-isocrystal $\End_{(s_\alpha)}(\bM')[\ivtd p]$. Therefore, for any  f-semiperfect $\Fpbar$-algebra $R$  we have \[
(\Bcris^+(R)\otimes_{\Zpbr}\bM^G_\cH)^{F=1} = \End_{(s_\alpha)}(\Bcris^+(R)\otimes_{\Zpbr}\bM')^{F=1};
\]
indeed, the positive slope part of $\End_{(s_\alpha)}(\Bcris^+(R)\otimes_{\Zpbr}\bM')$ does not have any non-zero $F$-invariance (\emph{cf.} the proof of Lemma~4.1.8 in \cite{CaraianiScholze:ShVar}). 

By \cite[Theorem~A]{ScholzeWeinstein:RZ}, for any f-semiperfect $\Fpbar$-algebra $R$  we have
\begin{multline*}
\wt\cH^G_b(R) = \Hom(\Qp/\Zp,(\cH^G_b)_R)[\ivtd p] 
= \Hom(\Bcris^+(R), \Bcris^+(R)\otimes_{\Zpbr}\bM^G_\cH)^{F=1} \\
= (\Bcris^+(R)\otimes_{\Zpbr}\bM^G_\cH)^{F=1} = \End_{(s_\alpha)}(\Bcris^+(R)\otimes_{\Zpbr}\bM')^{F=1}.
\end{multline*}
This shows that $\cH^G_b$ has the desired property.
\end{proof}


\begin{propsub}\label{prop:QEnd}
In the setting of Lemma~\ref{lem:TenHom} (e.g., under the assumptions (\ref{eqn:RunningAssumption1}, \ref{eqn:RunningAssumption2}) with $\BX_{\breve{\KK}}(b;\nu)\ne\emptyset$), the formal group scheme $\wt\cH^G_b$ only depends on $(G,b)$ up to isomorphism, not on the choice of $(s_\alpha)$ and $G\hra\GL(\Lambda)_{\Qp}$. The closed subgroup 
$\wt\cH^G_b\subset\wt \cH_b$ only depends on $\cG\hra \GL(\Lambda)$, not on $(s_\alpha)$. Furthermore,  the dimension of $\cH^G_b$  is $\langle 2\rho,\nu_{[b]}\rangle$, where $2\rho$ is the sum of all positive roots of $G_{\Qpbar}$ and $\nu_{[b]}$ is the dominant representative of the  conjugacy class of $\nu_b$. (Here, positive roots and dominant cocharacters are defined with respect to some choice of $B_{\Qpbr}\subset G_{\Qpbr}$ as in  \S\ref{subsec:GpTh}).
\end{propsub}
Although $\cH^G_b$ is well defined only up to isogeny,  the dimension of a $p$-divisible group is an isogeny invariant.
\begin{proof}
We use the notations as in \S\ref{subsec:GpTh} and the proof of Lemma~\ref{lem:TenHom}. 
We constructed the $p$-divisible group $\cH^G_b$ using the choice of $g\breve\KK\in X_{\breve\KK}(b;\nu)$ such that $\db{g\iv b\sigma(g)} = \db{\tilde w}$ for some $\tilde w$ (although its universal cover $\wt\cH^G_b$ does not depend on this choice; \emph{cf.} Lemma~\ref{lem:TenHom}). We choose a decent lift $\dot{\tilde w}$ of $\tilde w$ so that  its Newton cocharacter satisfies $\nu_{\dot{\tilde w}}\in X_*(S)_\Q$; \emph{cf.} Lemma~\ref{lem:Decency}. By modifying the choice of $B_{\Qpbr}$, we may assume that $\nu_{\dot{\tilde w}}$ is dominant (i.e., $\nu_{\dot{\tilde w}} = \nu_{[b]}$).
%

Note that $\rho:\cG\hra \GL(\Lambda)$ identifies
\[
\g_{\Qpbr}(:=\Lie G_{\Qpbr}) = \End_{(s_\alpha)}(\bM')[\ivtd p]
\]
as a Lie subalgebra of $\gl(\Lambda)_{\Qpbr} =  \End(\bM')[\ivtd p]$. Therefore, we can also view $\bM^G_\cH[\ivtd p]$ as the non-positive slope part of the $F$-isocrystal $(\g_{\Qpbr}, (\ad \dot{\tilde w})\sig)$, where $\sig$ on $\g_{\Qpbr}$ fixes the $\Qp$-structure $\g:=\Lie G$. In particular, the isogeny class of $\cH^G_b$ is determined by $(G,b)$, and the inclusion  $\cH^G_b\hra \cH_b$ only depends on $\cG\subset\GL(\Lambda)$, not on the choice of tensors $(s_\alpha)$. Since $\wt\cH^G_b$ only depends on the isogeny class of $\cH^G_b$, we obtain the independence claims in the statement.

Let us now describe the slope decomposition of the isocrystal
\[
(\g_{\Qpbr}, (\ad \dot{\tilde w}) \sigma).
\]
As in \S\ref{subsec:GpTh}, we fix the maximal $\Qpbr$-split torus $S\subset G$ defined over $\Qp$. We choose a $\Qpbr$-rational borel subgroup $B_{\Qpbr}\subset G_{\Qpbr}$ so that $\nu_{\dot{\tilde w}}$ is dominant.
With respect to this choice, we obtain the set of relative roots $\Phi_0$  for $G_{\Qpbr}$, and the subset of positive relative roots $\Phi_0^+$. We also have the following decomposition:
\[
\g_{\Qpbr} = \mathfrak{t} \oplus \left(\bigoplus_{\alpha_0\in\Phi_0} \mathfrak{u}_{\alpha_0}\right),
\]
where $\mathfrak{t} = \Lie(T_{\Qpbr})$ is the Lie algebra of $T=\cZ_G(S)$ over $\Qpbr$, and $\mathfrak u_{\alpha_0}$  is the $\alpha_0$-eigenspace with respect to the adjoint action of $S(\Qpbr)$.

To obtain the slope decomposition for $(\g_{\Qpbr}, (\ad \dot{\tilde w})\sigma)$, it suffices to get the slope decomposition for the $\Qpbr$-$\sigma^r$-space $(\g_{\Qpbr},((\ad \dot{\tilde w})\sigma)^r)$ for some suitable $r$. We may choose $r$ so that $(\dot{\tilde w}\sigma)^r = p^{r\nu_{[b]}}\sigma^r \in T(\Qpbr)\rtimes\sigma^\Z$ and the borel subgroup $B_{\Qpbr}\subset G_{\Qpbr}$ is actually defined over $\Q_{p^r}$. Note that the $\Q_{p^r}$-rationality of $B_{\Qpbr}$ ensures that each $\mathfrak{u}_{\alpha_0}$ is naturally defined over $\Q_{p^r}$; in particular, they are stable under $\sigma^r$. Furthermore, since $((\ad \dot{\tilde w}) \sigma)^r = \ad(p^{r\nu_{[b]}})\sigma^r$, it follows that
\begin{itemize}
\item for any $\alpha_0\in\Phi_0$, $((\ad \dot{\tilde w})\sigma)^r$ acts on $\mathfrak{u}_{\alpha_0}$ by $p^{\langle \alpha_0,r\nu_{[b]}\rangle}\sigma^r$;
\item $((\ad \dot{\tilde w})\sigma)^r$ acts on $\mathfrak{t}$ as $\sigma^r$.
\end{itemize}
Therefore, it follows that $\mathfrak{t}\subset \g_{\Qpbr}$ is a sub-isocrystal pure of slope $0$, and $\mathfrak{u}_{\alpha_0}$ is contained in the sub-isocrystal pure of slope $\langle \alpha_0,\nu_{[b]}\rangle$.  By dominance of $\nu_{[b]}$, $\langle \alpha_0,\nu_b\rangle$ is negative only if $\alpha_0$ is a negative relative root. Therefore, we have
\[
\bM^G_\cH = \left(\bigoplus_{\alpha_0\in\Phi_0^+} \mathfrak{u}_{-\alpha_0}\right)\oplus\left(\bigoplus_{\alpha_0\in\Phi_0^+ \text{ s.t. } \langle \alpha_0,\nu_{[b]}\rangle=0} \mathfrak{u}_{\alpha_0}\right) \oplus \mathfrak{t}.
\]
Note that only $\mathfrak{u}_{-\alpha_0}$ with $\langle \alpha_0,\nu_{[b]}\rangle > 0$ contributes to $\dim\cH^G_b$, and we have
\begin{equation}\label{eqn:dim}
\dim\cH^G_b = \sum_{\alpha_0\in\Phi_0^+} \langle \alpha_0, \nu_{[b]} \rangle \dim \mathfrak{u}_{\alpha_0}
\end{equation}
(Note that $\dim \mathfrak{u}_{\alpha_0} = \dim \mathfrak{u}_{-\alpha_0}$.)

Now let $\Phi^+\subset X^*(T)$ denote the set of positive (absolute) roots for $G_{\ol{\breve\Q}_p}$ (where the positivity is defined by $B_{\Qpbr}$), and for $\alpha\in\Phi^+$ we write $\mathfrak{u}_\alpha$ for the $\alpha$-eigenspace for the adjoint action of $T(\ol{\breve\Q}_p)$ on $\g_{\ol{\breve\Q}_p}$, which is necessarily one-dimensional. For any positive relative root $\alpha_0\in\Phi_0$, we have
\[
\mathfrak{u}_{\alpha_0}\otimes_{\Qpbr}\ol{\breve\Q}_p = \bigoplus_{\alpha\in\Phi^+\text{ s.t. } \alpha|_{S(\Qpbr)} = \alpha_0}\mathfrak{u}_\alpha;
\]
In particular, $\dim\mathfrak{u}_{\alpha_0}$ is equal to the number of absolute roots $\alpha$ with $\alpha|_{S(\Qpbr)} = \alpha_0$.
Since we have  $\langle \alpha,\nu_{[b]}\rangle = \langle \alpha|_{S(\Qpbr)},\nu_{[b]}\rangle$ for any absolute root $\alpha$ (as $\nu_{[b]}\in X_*(S)_{\Q,+}$), formula (\ref{eqn:dim}) becomes
\[
\dim\cH^G_b = \sum_{\alpha\in\Phi^+} \langle \alpha, \nu_{[b]} \rangle = \langle \sum_{\alpha\in\Phi^+}\alpha, \nu_{[b]} \rangle =:\langle 2\rho,\nu_{[b]}\rangle,
\]
as desired.
\end{proof}

\subsection{Self quasi-isogeny groups}\label{subsec:Qisg}
Let $\Qisg(\BX_b)$ denote the sheaf of groups on  $\Nilp_{\Zpbr}$, sending $R$ to the group of self quasi-isogenies  of $(\BX_b)_{R/p}$. Then $\Qisg(\BX_b)$ can be represented by a formal group scheme over $\Spf\Zpbr$; \emph{cf.} \cite[Lemma~4.2.10]{CaraianiScholze:ShVar}. Indeed, we have a closed immersion of formal schemes $\Qisg(\BX_b)\hra(\wt\cH_b)^2$ given by $\gamma\mapsto(\gamma,\gamma\iv)$, inducing a natural bijection
\[
\Qisg(\BX_b)(R) \riso \{(\gamma,\gamma')\in(\wt\cH_b)^2(R):\  \gamma\circ\gamma' =\gamma'\circ\gamma = \id\}
\]
 for any $R\in\Nilp_{\Zpbr}$. (Here, we identify $(\wt\cH_b)(R) = \End((\BX_b)_{R/p})[\ivtd p]$.) Similarly, we can make the following definition.
\begin{defnsub}\label{def:Qisg}
In the setting of Lemma~\ref{lem:TenHom} (e.g., under the assumptions (\ref{eqn:RunningAssumption1}, \ref{eqn:RunningAssumption2}) with $\BX_{\breve{\KK}}(b;\nu)\ne\emptyset$), we define the closed formal subgroup scheme $\Qisg_G(\BX_b)\subset \Qisg(\BX_b)$ over $\Spf\Zpbr$, such that the underlying formal subscheme is
\[
\Qisg_G(\BX_b) = \Qisg(\BX_b)\times_{(\wt\cH_b)^2}(\wt\cH^G_b)^2.
\]
Note that we have a closed immersion of formal schemes $\Qisg_G(\BX_b)  \hra (\wt\cH^G_b)^2$,  inducing a natural bijection
\[
\Qisg_G(\BX_b)(R) \riso \{(\gamma,\gamma')\in(\wt\cH^G_b)^2(R):\  \gamma\circ\gamma' =\gamma'\circ\gamma = \id\}
\]
 for any $R\in\Nilp_{\Zpbr}$.
\end{defnsub}
The following is a straightforward corollary of Lemma~\ref{lem:TenHom}.
\begin{corsub}\label{cor:TenHom}
Let  $R\in\Nilp_{\Zpbr}$, and assume that $R/p$ is f-semiperfect. Then $\gamma\in\Qisg(\BX_b)$ lies in
 $\Qisg_G(\BX_b)(R)$ if and only if  $\gamma_{R/p}: (\BX_b)_{R/p}\dra (\BX_b)_{R/p}$ preserves the tensors $(s_\alpha)$.
\end{corsub}

For the rest of this section, we will describe $\Qisg_G(\BX_b)$ as a formal scheme, generalising \cite[Proposition~4.2.11]{CaraianiScholze:ShVar}.

Recall that $J_b(\Qp)\subset G(\Qpbr)$ is the stabiliser of $b\sigma$ via the conjugation action of $G(\Qpbr)$ on $G(\Qpbr)\rtimes \sigma^\Z$. Similarly, let $J^\GL_{b}(\Qp)\subset \GL(\Lambda)(\Qpbr)$ denote the stabiliser of $b\sigma\in \GL(\Lambda)(\Qpbr)\rtimes \sigma^\Z$, viewing $G$ as a subgroup of $\GL(\Lambda)$ using the fixed embedding. Note that $J_b(\Qp)$ and $J^\GL_b(\Qp)$ are the group of $\Qp$-points of some reductive group over $\Qp$, so they are equipped with a natural locally profinite topology.

By Dieudonn\'e theory over $\Fpbar$, we can interpret $J^\GL_b(\Qp)$ as the group of self quasi-isogenies of $\BX_b$, and $J_b(\Qp)\subset J^\GL_b(\Qp)$ as the subgroup of tensor-preserving self quasi-isogenies.

Let $\underline{J_b(\Qp)}$ and $\underline{J^\GL_b(\Qp)}$ respectively denote the formal group schemes over $\Spf\Zpbr$ associated to the locally profinite groups $J_b(\Qp)$ and $J^\GL_b(\Qp)$. (To describe the underlying formal scheme of $\underline{J_b(\Qp)}$, for any open compact subset $U\subset J_b(\Qp)$ such that $U$ is the  limit of the projective system of finite sets $\set{U_m}$, we have an formal open subscheme $\underline U \subset\underline{J_b(\Qp)}$ with $\underline U = \varprojlim\underline{U_m}$.)

Recall that $\Qisg(\BX_b)(\Fpbar) = J^\GL_b(\Qp)$, so we have a natural map $\underline{J^\GL_b(\Qp)}\ra\Qisg(\BX_b)$ of group-valued sheaves on $\Nilp_{\Zpbr}$  as follows: for  $R\in \Nilp_{\Zpbr}$, we send a self quasi-isogeny $\gamma:\BX_b\dra\BX_b$ to $\gamma_{R/p}$. (This a priori defines a map from  the constant ``discrete'' group associated to $J^\GL_b(\Qp)$, but one can easily see that this map factors through the constant locally profinite group $\underline{J^\GL_b(\Qp)}$.) It is shown in \cite[Proposition~4.2.11]{CaraianiScholze:ShVar} that we have a natural section
\begin{equation}
\Qisg(\BX_b)\ra \underline{J^\GL_b(\Qp)}
\end{equation}
with all the fibres  isomorphic to $\Spf\Zpbr [[ x_1^{1/p^\infty},\cdots, x_{d'}^{1/p^\infty} ]]$. Here, $d' = \langle 2\rho',\nu_b\rangle$, where $2\rho'$ is the sum of all the positive roots of $\GL(\Lambda)_{\Qpbar}$.

Since we have $\Qisg_G(\BX_b)(\Fpbar) = J_b(\Qp)$, we also have a natural map of group-valued sheaves $\underline{J_b(\Qp)}\ra \Qisg_G(\BX_b)$.
Furthermore, we have the following generalisation of \cite[Proposition~4.2.11]{CaraianiScholze:ShVar}:
\begin{propsub}\label{prop:Qisg}
 The formal subgroup scheme $\Qisg_G(\BX_b)\subset \Qisg(\BX_b)$ is independent of the choice of tensors $(s_\alpha)\subset\Lambda^\otimes$, and the formal group scheme $\Qisg_G(\BX_b)$ only depends on $(G,b)$ up to isomorphism. Furthermore, there exists a natural section of formal group schemes
\[
\Qisg_G(\BX_b) \ra \underline{J_b(\Qp)},
\]
with  all the fibres isomorphic to  $\Spf\Zpbr [[ x_1^{1/p^\infty},\cdots, x_d^{1/p^\infty}]]$ as a formal scheme. Here, $d = \langle 2\rho,\nu_{[b]}\rangle$, where $2\rho$ is the sum of all the positive roots of $G_{\Qpbar}$ and $\nu_{[b]}$ is the dominant representative of the conjugacy class of $\nu_b$. 
\end{propsub}
\begin{rmksub}\label{rmk:Qisg}
We set $\Qisg^\circ_G(\BX_b):=\ker\big( \Qisg_G(\BX_b) \ra \underline{J_b(\Qp)}\big)$. 
Then Proposition~\ref{prop:Qisg} implies that
\[
\Qisg_G(\BX_b) = \Qisg^\circ_G(\BX_b) \rtimes \underline{J_b(\Qp)}.
\]
\end{rmksub}
\begin{proof}[Proof of Proposition~\ref{prop:Qisg}]
The independence of the choice of tensors $(s_\alpha)$ follows from the same property for $\cH^G_b$. Clearly, the natural section $\Qisg(\BX_b)\thra \underline{J_b^\GL(\Qp)}$ restricts to $\Qisg_G(\BX_b)\thra \underline{J_b(\Qp)}$.

Let $\cH^{G,\circ}_b$ be the connected part of  $\cH^{G}_b$, and let $\wt\cH^{G,\circ}_b$ denote the unique lift over $\Zpbr$ of the universal cover $ \varprojlim_{[p]}\cH^{G,\circ}_b$ of  $\cH^{G,\circ}_b$.
We now claim that the following map
\[
\xymatrix@1{
\Qisg_G(\BX_b) \ar[rrr]^-{\gamma\mapsto(\gamma,\gamma\iv)} &&& (\wt\cH^G_b)^2 \ar[rrr]^-{(\gamma_1,\gamma_2)\mapsto \gamma_1-\id} &&& \wt\cH^G_b
}\]
induces an isomorphism $\Qisg^\circ_G(\BX_b)\riso\wt\cH^{G,\circ}_b$ of formal schemes, where $\Qisg^\circ_G(\BX_b):=\ker\big( \Qisg_G(\BX_b) \ra \underline{J_b(\Qp)}\big)$. Indeed, this isomorphism $\Qisg^\circ(\BX_b)\riso \wt\cH_b^\circ$ in the case of $\GL(\Lambda)_{\Qp}$ can be read off from the proof of Proposition~4.2.11 in \cite{CaraianiScholze:ShVar}. Therefore, for any $\gamma\in\wt\cH^{G,\circ}_b(R)$ where $p$ is nilpotent in $R$ and $R/p$ is f-semiperfect, $(\id+\gamma)_{R/p}$ is a (necessarily tensor-preserving) self quasi-isogeny of $(\BX_b)_{R/p}$, so it defines a section in $\Qisg^\circ_G(\BX_b)(R)$. Conversely, the isomorphism $\Qisg^\circ(\BX_b)\riso \wt\cH_b^\circ$ clearly takes $\Qisg^\circ_G(\BX_b)$ to $\wt\cH^{G,\circ}_b$. Hence, we obtain the desired isomorphism $\Qisg^\circ_G(\BX_b)\riso\wt\cH^{G,\circ}_b$.

Now, the description of the fibre of $\Qisg_G(\BX_b)\thra \underline{J_b(\Qp)}$ follows from Proposition~\ref{prop:QEnd} and \cite[Proposition~3.1.3(iii)]{ScholzeWeinstein:RZ}.
\end{proof}

\section{Review of Kisin-Pappas deformation rings}\label{sec:KP}
From now on, we always assume that $p>2$.

\subsection{Review of Dieudonn\'e display theory}
We briefly review and set up the notation for the theory of Dieudonn\'e displays for \emph{$p$-torsion free} complete local noetherian rings with residue field $\Fpbar$, following \cite[\S3.1]{KisinPappas:ParahoricIntModel}. For more standard references for Dieudonn\'e display theory, we refer to Zink's original paper \cite{Zink:DieudonnePDivGpCFT} and Lau's paper \cite{Lau:DisplayCrystals}.

Let $R$ be a complete local noetherian ring with residue field $\Fpbar$. In \cite{Zink:DieudonnePDivGpCFT}, Zink introduced a $p$-adic subring $\Zink(R)$ of the ring of Witt vectors $W(R)$, which fits in the following short exact sequence:
\[
\xymatrix@1{
0 \ar[r] & \widehat W(\m_R) \ar[r] & \Zink (R) \ar[r] & W(\Fpbar) \ar[r] &0,
}
\]
where $\widehat W(\m_R):=\{(a_i)\in W(\m_R)|\ a_i\to 0\}$.  If $p>2$ then $\Zink(R)$ is stable under the Frobenius and Verschiebung operators of $W(R)$. Recall that $\Zink(R)$ is $p$-torsion free if $R$ is $p$-torsion free (since in that case $W(R)$ is $p$-torsion free).

Let $\BI_R$ denote the kernel of the natural projection $\Zink(R)\thra R$, which coincides with the injective image of the Verschiebung operator (if $p>2$). We let $\sigma: \Zink(R)\ra\Zink (R)$ denote the Witt vector Frobenius map. Also if $p>2$ then $\BI_R$ is stable under the natural divided power structure on the kernel of $W(R)\thra R$, and $\Zink(R)$ is a $p$-adic divided power thickening of $R$. (In particular, one can evaluate a crystal over $R$ at $\Zink(R)$.)

\begin{defnsub}\label{def:display}
A \emph{Dieudonn\'e display} over $R$ is a tuple
\[
(M, M_1, \Phi, \Phi_1),
\]
where
\begin{enumerate}
\item $M$ is a finite free $\Zink(R)$-module;
\item $M_1\subset M$ is a $\Zink(R)$-submodule containing $\BI_R M$, such that $M/M_1$ is free over $R$;
\item $\Phi:M\ra M$ and $\Phi_1:M_1\ra M$ are $\sigma$-linear morphism such that $p\Phi_1 = \Phi|_{M_1}$ and the image of $\Phi_1$ generates $M$.
\end{enumerate}
The \emph{Hodge filtration} associated to the display is $M_1/\BI_R M \subset M/\BI_R M$ viewed as the $0$th filtration (so $M/M_1$ is viewed as the $(-1)$th grading).
\end{defnsub}

Given a Dieudonn\'e display $(M,M_1,\Phi,\Phi_1)$ over $R$, we let $\wt M_1$ denote the image of $\sigma^*M_1$ in $\sigma^*M$. Then the linearisation of $\Phi_1$ induces  the following isomorphism
\[
\Psi:\wt M_1 \riso M.
\]
If $\Zink(R)$ is $p$-torsion free (for example, if $R$ is $p$-torsion free), then given $M$ and $M_1\subset M$ where $M/M_1$ is projective over $R$, giving a $\sig$-linear map $\Phi$ and $\Phi_1$ that makes $(M,M_1,\Phi,\Phi_1)$ a Dieudonn\'e display is equivalent to giving an isomorphism $\Psi$ as above; \emph{cf.} \cite[Lemma~3.1.5]{KisinPappas:ParahoricIntModel}.

For any complete local noetherian ring $R$ with perfect residue field of characteristic $p>2$,
Zink \cite{Zink:DieudonnePDivGpCFT} constructed a natural equivalance between the category of $p$-divisible groups over $R$ and the category of Dieudonn\'e displays over $R$. We can describe this (covariant) equivalence of categories when $\Zink(R)$ is $p$-torsion free in terms of crystalline Dieudonn\'e theory using Lau's result \cite{Lau:DisplayCrystals}, as follows.

Assume that $\Zink(R)$ is $p$-torsion free (for example,  $p$-torsion free $R$). Then for a $p$-divisible group $X$ over $R$, the (covariantly) associated Dieudonn\'e display is given by the following data:
\begin{enumerate}
\item $M:= (\DD(X)(\Zink(R)))^*$ and $M_1:=\ker \left(M \thra (\DD(X)(R))^* \thra \Lie(X)\right)$, where $\DD(X)$ is the contravariant Dieudonn\'e crystal;
\item The crystalline Frobenius $F:\sigma^*M[\ivtd p]\ra M[\ivtd p]$ restricts to an isomorphism $\Psi:\wt M_1\ra M$.
\end{enumerate}

 \begin{rmksub}\label{rmk:DisplayModp}
 Let $X$ be a $p$-divisible group over a $R$ as above (with the extra assumption that $R$ is $p$-torsion free). Let  $(M,M_1,\Phi,\Phi_1)$ denote the associated Dieudonn\'e display.
 Note that the underlying crystal $\DD(X)$ together with the crystalline Frobenius only depends on $X_{R/p}$. Therefore, the underlying $\Zink(R)$-module $M$ and $\Psi:(\sigma^*M)[\ivtd p]\riso M[\ivtd p]$ are determined by $X_{R/p}$.
%
 \end{rmksub}

\subsection{Review of Kisin-Pappas deformation theory}\label{subsec:DeforKP}
We begin with the review of the deformation theory of $p$-divisible groups with tensors in \cite[\S3]{KisinPappas:ParahoricIntModel}. Let us first consider the case of $\GL_n$
 We choose $b\in \GL_n(\Qpbr)$ such that there exists a $p$-divisible group $\BX:=\BX_b$ over $\Fpbar$ with $\big(\DD(\BX)(\Zpbr)\big)^* = (\Zpbr\otimes_{\Zp}\Lambda,b\sig)=:\bM (=\bM_b)$, where $\Lambda = \Zp^n$.
We also view $\bM$ as a Dieudonn\'e display $(\bM,\bM_1,\Phi,\Phi_1)$, where $\bM_1:=(b\sig)\iv(\bM) \subset \bM$, $\Phi_1 = b\sigma$ and $\Phi = pb\sigma$. Since we have $\sigma^*\bM_1\cong \wt\bM_1$, it follows that $\Psi:\wt\bM_1\riso \bM$ coincides with the linearisation of $\Phi_1$.

Note that $\bM/\bM_1 = \Lie\BX$ and $\bM/p\bM \thra \bM/\bM_1$ recovers the natural map $\big(\DD(\BX)(\Fpbar)\big)^*\thra \Lie\BX$ defining the Hodge filtration. Let $R_\GL$ denote the completed local ring of a suitable grassmannian variety over $\Zpbr$ at the point corresponding to $\bM/p\bM \thra \bM/\bM_1$.

In \cite[\S3.1]{KisinPappas:ParahoricIntModel}, Kisin and Pappas constructed a Dieudonn\'e display
\[(M_\GL,M_{\GL,1},\Phi,\Phi_1)\]
over $R_\GL$, which is a universal deformation of $\bM$ with the following properties:
\begin{enumerate}
\item We have $M_\GL = \Zink(R_\GL)\otimes_{\Zpbr}\bM$ as a $\Zink(R_\GL)$-module.
\item Identifying $R_{\GL}$ with the universal deformation ring of the Hodge filtration $\bM/p\bM \thra \bM/\bM_1$ (i.e., the completed local ring of some  grassmannian), 
the  Hodge filtration $M_\GL\otimes_{\Zink(R_\GL)}R_\GL \thra M_\GL/M_{\GL,1}$ is the universal deformation of $\bM/p\bM \thra \bM/\bM_1$.
\item The identification of the reduced tangent spaces of the deformation functor and of $\Spf R_\GL$ is compatible with the one provided by the Grothendieck-Messing deformation theory (in the sense of \cite[Lemma~3.1.15]{KisinPappas:ParahoricIntModel}).
\end{enumerate}

\begin{defnsub}\label{def:LocShDatum}
We recall, from \cite[\S3.2.5]{KisinPappas:ParahoricIntModel}, the definition of a quotient $R_\cG$ of $R_\GL\otimes_{\Zpbr}\fo_{\breve E}$ (where $\fo_{\breve E}$ is a suitable finite extension of $\Zpbr$). Let $\cG_{/\Zp}$ be a Bruhat-Tits integral model of $G_{/\Qp}$ as in Definition~\ref{def:KP}. We additionally assume that $G$ splits after a tame extension of $\Qp$, $p$ does not divide the order of $\pi_1(G^\der)$ (\emph{cf.} (\ref{eqn:RunningAssumption1})), and the adjoint group of $G$ does not have a factor of type $E_8$.\footnote{We assume that $p\nmid |\pi_1(G^\der)|$ to ensure that the local model  $\mathrm{M}^\loc_{\cG^\circ,\set\mu}$ is normal (so the ``deformation ring'' $R_\cG$ is normal).
The rest of the assumptions are made because we need to have \cite[Proposition~1.4.3]{KisinPappas:ParahoricIntModel} for the deformation theory, which is proved under the assumption that $G$ splits after a tame extension and the adjoint group of $G$ has no factor of type $E_8$.}

We choose a  \emph{local Shimura datum}
\[(G,[b],\set{\mu})\]
in the sense of \cite[Definition~5.1]{RapoportViehmann:LocShVar}; in other words, $\set\mu$ is a $G(\ol\Qpbr)$-conjugacy class of minuscule cocharacters of $G$, and $[b]$ is a $\sigma$-$G(\Qpbr)$ conjugacy class that is neutral acceptable for $\set\mu$ in the sense of  \cite[Definition~2.3]{RapoportViehmann:LocShVar}. Let $\breve E\subset\ol\Qpbr$ denote the field of definition of the conjugacy class $\set\mu$ over $\Qpbr$, which is a finite extension of $\Qpbr$.

We choose $b\in[b]$ and a closed immersion $\cG\hra \GL(\Lambda)$ such that the following properties hold:
\begin{enumerate}
\item The image of $G = \cG_{\Qp}$ in $\GL(\Lambda)_{\Qp}$ contains scalar matrices.
\item 
The closed immersion $\cG\hra \GL(\Lambda)$  sends $\set \mu$ to $\set{\nu_d}$  for some $d$, where $\nu_d$ is the minuscule cocharacter of $\GL_n$ as in Example~\ref{exa:SlopeFil}; \emph{cf.} (\ref{eqn:RunningAssumption2}). 
\item 
There exists a $p$-divisible group $\BX:=\BX_b$ over $\Fpbar$ with $\big(\DD(\BX)(\Zpbr)\big)^* = (\Zpbr\otimes_{\Zp}\Lambda,b\sig)=:\bM (=\bM_b)$ as a virtual $F$-crystal. Furthermore, there exists a cocharacter $\mu_\BX:\Gm\to\GL(\Lambda)_{\fo_K}$ over some finite extension $\fo_K$ of $\fo_{\breve E}$, such that its generic fibre factors through $G_K$ and its special fibre induces the Hodge filtration of $\BX$.
%
\end{enumerate}
Under these assumptions, the identity right coset $\breve\KK \in G(\Qpbr)/\breve\KK$ belongs to $X_{\breve\KK}(b;\sigma^*\mu)$; \emph{cf.} \cite[\S5.4]{Zhou:ModpPts}. 
In this case, we have constructed the formal closed subgroup scheme $\Qisg_G(\BX)\subset \Qisg(\BX)$ (\emph{cf.} Definition~\ref{def:Qisg}), which can be applied since we have (\ref{eqn:RunningAssumption1}, \ref{eqn:RunningAssumption2}) and $X_{\breve\KK}(b;\sigma^*\mu)\ne 0$ (with $\nu=\sig^*\mu$).

Such $b\in[b]$ and $\cG\hra \GL(\Lambda)$ exist if the local Shimura datum $(G,[b],\set\mu)$ comes from a Hodge-type Shimura datum (using \cite[Corollary~2.3.16]{KisinPappas:ParahoricIntModel}), but they may not exist for an arbitrary local Shimura datum.

%
\end{defnsub}

Recall that we have a natural closed immersion $\mathrm{M}^\loc_{\cG^\circ,\set\mu}\hra (\mathrm{M}^\loc_{\GL(\Lambda),\set{\nu_d}})_{\fo_{\breve E}}$ of Pappas-Zhu local models (\emph{cf.} (\ref{eqn:EmbLocMod})),  and $\mathrm{M}^\loc_{\GL(\Lambda),\set{\nu_d}}$ is  the grassmannian  over $\Zpbr$ classifying rank-$d$ quotients. Furthermore, the Hodge filtration of $\BX$ lies in the image of $\mathrm{M}^\loc_{\cG^\circ,\set\mu}$; indeed, the filtration defined by the cocharacter $\mu_{\BX}$ gives an $\fo_K$-point of the local model $\mathrm{M}^\loc_{\cG^\circ,\set\mu}$; \emph{cf.} \cite[\S3.2.5]{KisinPappas:ParahoricIntModel}.

We identify $R_\GL$ as the completed local ring of the grassmannian $\mathrm{M}^\loc_{\GL(\Lambda),\set\mu}$ at the $\Fpbar$-point corresponding to the Hodge filtration of $\BX$. We set $R_\cG$ as the completion of $\mathrm{M}^\loc_{\cG^\circ,\set\mu}$ at the same $\Fpbar$-point.
Using the assumption that $p$ does not divide the order of $\pi_1(G^\der)$, it follows that $R_\cG\otimes_{\fo_{\breve E}}\fo_K$ is normal for any finite extension $K$ of $\breve E$; \emph{cf.} \cite[Corollary~2.1.3]{KisinPappas:ParahoricIntModel}.

We choose finitely many tensors $(s_\alpha)\subset \Lambda^\otimes$ whose pointwise stabiliser is (the image of) $\cG$ inside $\GL(\Lambda)$. Using $\bM = \Zpbr\otimes_{\Zp}\Lambda$, we view
$(s_\alpha)\subset\bM^\otimes$.
Furthermore, by viewing $\wt\bM_1 \subset (\sigma^*\bM)[\ivtd p] =\Qpbr\otimes_{\Zp}\Lambda$, we may view $(s_\alpha)\subset (\wt\bM_1)^\otimes[\ivtd p]$. And since $\Phi_1 = b\sigma$ with $b\in G(\Qpbr)$, it follows that  $(s_\alpha)\subset(\wt\bM_1)^\otimes$, and $\Psi:\wt\bM_1\riso \bM$ preserves the tensors $(s_\alpha)$. (Here, $\wt\bM_1$ is the image of $\sigma^*\bM_1$ in $\sigma^*\bM$, and $\Psi$ is the isomorphism induced by the linearisation of $\Phi_1$.)

Let $(M_\cG, M_{\cG,1}, \Phi,\Phi_1)$ denote the base change of the universal deformation of displays $(M_\GL,M_{\GL,1},\Phi,\Phi_1)$ to $R_\cG$. Using $M_\cG = \Zink(R_\cG)\otimes_{\Zpbr}\bM$, we view  $(s_\alpha)$ as elements in  $  M_\cG^\otimes$ or in $(\sigma^* M_\cG)^\otimes$. As before, we denote by $\wt M_{\cG,1}$ the image of $\sigma^* M_{\cG,1}$ in $\sigma^* M_\cG$, and $\Psi:\wt M_{\cG,1}\riso M_\cG$ the isomorphism induced by the linearisation of $\Phi_1$. 

The following proposition can be deduced from \cite[\S3.2]{KisinPappas:ParahoricIntModel}.
\begin{propsub}
In the setting of Definition~\ref{def:LocShDatum}, the tensors $(s_\alpha)\subset (\sigma^* M_\cG)^\otimes$ lie in $ \wt M_{\cG,1}^\otimes$. Furthermore, $\Psi:\wt M_{\cG,1}\riso M_\cG$ preserves the tensors $(s_\alpha)$.
\end{propsub}
\begin{proof}
The claim that $(s_\alpha)\subset  \wt M_{\cG,1}^\otimes$ is proved in \cite[Corollary~3.2.11]{KisinPappas:ParahoricIntModel}. Note that for any $\xi:R_\cG\ra \fo_K$ for any finite extension $\fo_K$ of $\fo_{\breve E}$, the scalar extension $\Psi$ via $\xi$  preserves the tensors $(s_\alpha)$ by \cite[Proposition~3.2.17(2)]{KisinPappas:ParahoricIntModel}.  Since $R_\cG[\ivtd p]$ is Jacobson, it follows that $\Psi$ preserves the tensor after the scalar extension to $W(R_\cG[\ivtd p])$. As $\Zink(R_\cG)$ is a subring of $W(R_\cG[\ivtd p])$, we conclude that $\Psi:\wt M_{\cG,1}\riso M_\cG$ preserves the tensors $(s_\alpha)$.
\end{proof}

For a finite extension $\fo$ of $\fo_{\breve E}$, one can characterise which deformation $X$ over $\fo$ defines an $\fo$-point of $R_\cG$ via the lifts of crystalline tensors $(s_\alpha)$; \emph{cf.} \cite[Proposition~3.2.17]{KisinPappas:ParahoricIntModel}. We will recall the statement in  Proposition~\ref{prop:KPDeforThy}. For this, we need some preparation.

Let $\xi:R_\GL \ra \fo$ where $\fo$ is a finite extension of $\fo_{\breve E}$, and let $(M_\xi,M_{\xi,1},\Psi)$ denote deformation of $(\bM, \bM_1,\Psi)$ corresponding to $\xi$. Then by \cite[Lemma~3.1.17]{KisinPappas:ParahoricIntModel}, there exists a unique $\Psi$-equivariant isomorphism
\begin{equation}\label{eqn:IsocRig}
\Zink(\fo)\otimes_{\Zpbr}\bM [\ivtd p] \riso M_\xi[\ivtd p]
\end{equation}
which reduces to the identity map on $\bM_\xi[\ivtd p]$ after the scalar extension by $\Zink(\fo)[\ivtd p] \thra \Zink(\Fpbar)[\ivtd p] = \Qpbr$. Alternatively, one can obtain this isomorphism as follows. Let $X_\xi$ be the $p$-divisible group associated to $(M_\xi,M_{\xi,1},\Psi)$, and consider the unique quasi-isogeny $\iota_\xi:X_{\xi,\fo/p}\dra \BX_{\fo/p} $ lifting the identity map on $\BX$. Then $\iota_\xi$ induces the isomorphism of $F$-isocrystals $\DD(X_{\xi,\fo/p})^*[\ivtd p] \riso \DD(\BX_{\fo/p})^*[\ivtd p]$, and the map on the $\Zink(\fo)$-sections coincides with the isomorphism (\ref{eqn:IsocRig}).

Using this isomorphism~(\ref{eqn:IsocRig}), we obtain the tensors
\begin{equation}\label{eqn:IsocTensors}
(s_\alpha) \subset (M_\xi)^\otimes[\ivtd p].
\end{equation}
By the discussion earlier, $(s_\alpha) \subset (M_\xi)^\otimes[\ivtd p]$ can also be obtained from the $\Zink(\fo)$-sections of the maps of $F$-isocrystals, also denoted by $(s_\alpha)$:
\begin{equation}\label{eqn:IsocTensorsMorph}
s_\alpha:\triv\ra  (\DD(X_{\xi,\fo/p})^*)^\otimes[\ivtd p].
\end{equation}

Note that the scalar extension by $\Zink(\fo)\thra\Zink(\Fpbar)$ sends the tensors $(s_\alpha)\subset M_\xi^\otimes[\ivtd p]$ (\ref{eqn:IsocRig}) to $(s_\alpha)\subset\bM^\otimes[\ivtd p]$, the tensors  that we started with.

\begin{propsub}\label{prop:KPDeforThy}
In the setting as above, $\xi:R_\GL\ra \fo$ factors through $R_\cG$ if and only if the following condition holds:
\begin{enumerate}
\item\label{prop:KPDeforThy:pInt}
The  tensors $(s_\alpha) \subset M_\xi^\otimes[\ivtd p]$, defined in (\ref{eqn:IsocTensors}), lie in $M_\xi^\otimes$.
\item\label{prop:KPDeforThy:Torsor}
There exists an $\Zink(\fo)$-linear isomorphism
\[
\Zink(\fo)\otimes_{\Zpbr}\bM \cong M_\xi,
\]
preserving the tensors  $(s_\alpha)$.
\item\label{prop:KPDeforThy:HodgeFil}
Let $K:=\fo[\ivtd p]$.
Using the isomorphism (\ref{eqn:IsocRig}) modulo $\BI_\fo$
\[K\otimes_{\Zp}\Lambda = K\otimes_{\Zpbr}\bM \cong K\otimes_{\Zink(\fo)}M_\xi,\]
the Hodge filtration $(M_{\xi,1}/\BI_\fo M_\xi)[\ivtd p]\subset K\otimes_{\Zp}\Lambda$ is given by a $G$-valued cocharacter $\mu_\xi:\Gm\ra G_K$ that belongs to the geometric conjugacy class of cocharacters $\set\mu$ associated to $b\in G(\Qpbr)$ in the sense of Definition~\ref{def:LocShDatum}.
\end{enumerate}
\end{propsub}

\begin{proof}
By \cite[Proposition~3.2.17, Lemma~3.2.13]{KisinPappas:ParahoricIntModel}, if $\xi$ factors through $R_\cG$ then the tensors $(s_\alpha)$ satisfy all the conditions listed in the statement; note that the tensors $(\tilde s_\alpha)\subset M_\xi^\otimes$ in \cite[Proposition~3.2.17]{KisinPappas:ParahoricIntModel} should coincide with our $(s_\alpha)$ in (\ref{eqn:IsocTensors}) by \cite[Lemma~3.2.13]{KisinPappas:ParahoricIntModel}. The converse follows from \cite[Proposition~3.2.17]{KisinPappas:ParahoricIntModel}.
\end{proof}
\begin{defnsub}
In the setting of Proposition~\ref{prop:KPDeforThy}, assume that $\xi:R_\GL\ra \fo$  satisfies Proposition~\ref{prop:KPDeforThy}(\ref{prop:KPDeforThy:HodgeFil}). (This is satisfied if $\xi$ factors through $R_\cG$.) Let $T(X_\xi)$ denote the integral Tate module of $X_\xi$. Then by the crystalline comparison isomorphism $\Bcris\otimes_{\Zp}T(X_\xi)\cong \Bcris\otimes_{\Zpbr}\bM$, we obtain  Galois-invariant tensors
\[(s_{\alpha,\et})\subset T(X_\xi)^\otimes[\ivtd p].\]
\end{defnsub}

Let us now recall the statement of \cite[Proposition~3.3.13]{KisinPappas:ParahoricIntModel}:
\begin{propsub}\label{prop:KPEtaleCycles}
The map $\xi:R_\GL\ra\fo$ factors through $\fo$ if the following conditions hold:
\begin{enumerate}
\item\label{prop:KPEtaleCycles:HodgeFil} The map $\xi$ satisfies Proposition~\ref{prop:KPDeforThy}(\ref{prop:KPDeforThy:HodgeFil}).
\item\label{prop:KPEtaleCycles:Torsor} We have $(s_{\alpha,\et})\subset T(X_\xi)^\otimes$, and there exists a $\Zpbr$-linear isomorphism $\bM \cong \Zpbr\otimes_{\Zp} T(X_\xi)$ matching $(s_\alpha)$ with $(s_{\alpha,\et})$.
\end{enumerate}
\end{propsub}

\begin{rmksub}\label{rmk:NonNeutral}
We stated Proposition~\ref{prop:KPEtaleCycles} in a slightly more restrictive form than  \cite[Proposition~3.3.13]{KisinPappas:ParahoricIntModel}; namely, Proposition~\ref{prop:KPEtaleCycles}(\ref{prop:KPEtaleCycles:Torsor}) is more stringent than \emph{loc. cit}. Indeed, assuming that $\xi:R_\cG\ra\fo$ satisfies that $(s_{\alpha,\et})\subset T(X_\xi)^\otimes$, it seems that the following $\Zp$-scheme
\[\Isom([\Lambda,(s_\alpha)], [T(X_\xi), (s_{\alpha,\et})])\]
may be a non-trivial $\cG$-torsor if $\cG\ne\cG^\circ$, so Proposition~\ref{prop:KPEtaleCycles}(\ref{prop:KPEtaleCycles:Torsor}) may not be satisfied. 

It is possible, with a little more work,  to improve Proposition~\ref{prop:KPEtaleCycles} to a necessary and sufficient condition.
On the other hand,  we are interested in the deformation ring $R_\cG$ that occurs as the completed local ring of some parahoric-level integral model of Hodge-type Shimura varieties constructed in \cite{KisinPappas:ParahoricIntModel}. In that case, we can arrange so that any $\xi:R_\cG\ra\fo$ satisfies Proposition~\ref{prop:KPEtaleCycles}(\ref{prop:KPEtaleCycles:Torsor}). (
See \cite[\S4.1.7]{KisinPappas:ParahoricIntModel} for more details.)
\end{rmksub}

\subsection{The action of $\Qisg^\circ_G(\BX)$ on $\Spf R_\cG$}\label{subsec:QisgRZ}
Recall that $\Spf R_\GL$ can be regarded as the completion of some Rapoport-Zink space. In other words, for any local $\Zpbr$-algebra $R$ where $p$ is nilpotent, $\Hom_{\Zpbr}(R_\GL, R)$ is in natural bijection with the isomorphism class of pairs $(X,\iota)$ where $X$ is a deformation of $\BX$ over $R$ and $\iota:X_{R/p}\dra \BX_{R/p}$ is a quasi-isogeny lifting $\id_{\BX}:\BX\ra\BX$. Therefore, we get a natural action of $\Qisg^\circ(\BX)$ on $\Spf R_\GL$ as follows: $\gamma\in \Qisg_G(\BX_b)(R)$ sends $(X,\iota)$ to $(X,\gamma\circ\iota)$. (Indeed, the full self quasi-isogeny group $\Qisg(\BX)$ naturally act on the Rapoport-Zink space, and it restricts to this  $\Qisg^\circ(\BX)$-action on the completion $\Spf R_\GL$.)

Recall that the construction $\Qisg_G(\BX)\subset \Qisg(\BX)$ can be applied to the setting of Definition~\ref{def:LocShDatum}. 
The goal of this section is to prove the following theorem:
\begin{thmsub}\label{thm:QisgActionRZ}
Assume that any $\xi: \Spf\fo\ra\Spf R_\cG$  satisfies Proposition~\ref{prop:KPEtaleCycles}(\ref{prop:KPEtaleCycles:Torsor}), where $\fo$ is a finite extension of $\fo_{\breve E}$.\footnote{As explained in Remark~\ref{rmk:NonNeutral}, this assumption can be arranged if $R_\cG$ came from some integral model of Shimura varieties constructed in \cite{KisinPappas:ParahoricIntModel}. 
}
Then the natural $\Qisg^\circ(\BX)$-action on $\Spf R_\GL$ restricts to a natural $\Qisg^\circ_G(\BX)$-action on $\Spf R_\cG$; in other words, the morphism of formal schemes
\begin{equation}\label{eqn:QisgActionRZ}
\Qisg_G^\circ(\BX) \times_{\Spf\Zpbr} \Spf R_\cG \ra\Spf (R_\GL\otimes_{\Zpbr}\fo_{\breve E}),
\end{equation}
defined by the natural $\Qisg^\circ(\BX)$-action on $\Spf R_\GL$, factors through $\Spf R_\cG$.
\end{thmsub}

We can deduce this theorem from the following proposition:

\begin{propsub}\label{prop:QisgActionRZ}
Let $\fo$ be a finite extension of $\fo_{\breve E}$, and assume that $\xi: \Spf\fo\ra\Spf R_\cG$ satisfies Proposition~\ref{prop:KPEtaleCycles}(\ref{prop:KPEtaleCycles:Torsor}). Then the following map
\[\xymatrix@1{
\Qisg_G^\circ(\BX)\times_{\Spf\Zpbr}\Spf\fo \ar[r]^-{(\id,\xi)} &
\Qisg_G^\circ(\BX) \times_{\Spf\Zpbr} \Spf R_\cG \ar[r]^-{\text{(\ref{eqn:QisgActionRZ})}}&
\Spf (R_\GL\otimes_{\Zpbr}\fo_{\breve E})
}\]
factors through $R_\cG$.
\end{propsub}
Granting Proposition~\ref{prop:QisgActionRZ}, one can deduce Theorem~\ref{thm:QisgActionRZ} as follows:
\begin{proof}[Proof of Theorem~\ref{thm:QisgActionRZ}]
We write the underlying formal scheme for $\Qisg^\circ_G(\BX)$ as $\Spf S^\infty_G$ with $S^\infty_G = \Zpbr[[ x_1^{1/p^\infty},\cdots,x_d^{1/p^\infty}]]$; \emph{cf.} Proposition~\ref{prop:Qisg}. Then the map (\ref{eqn:QisgActionRZ}) can be written as
\begin{equation}\label{eqn:QisgActionRing}
R_\GL\otimes_{\Zpbr}\fo_{\breve E} \ra R_\cG\wh\otimes_{\Zpbr}S^\infty_G = R_\cG[[ x_1^{1/p^\infty},\cdots,x_d^{1/p^\infty}]].
\end{equation}
We want to show that this map factors through the quotient $R_\cG$.

By Proposition~\ref{prop:QisgActionRZ}, the following map 
\begin{equation}\label{eqn:QisgActionPt}
\xymatrix@1{
R_\GL\otimes_{\Zpbr}\fo_{\breve E} \ar[r]^-{\text{(\ref{eqn:QisgActionRing})}}&
R_\cG\wh\otimes_{\Zpbr}S^\infty_G \ar[r]^-{\xi} &
\fo\wh\otimes_{\Zpbr}S^\infty_G = \fo[[ x_1^{1/p^\infty},\cdots,x_d^{1/p^\infty}]]
}
\end{equation}
factors through $R_\cG$ for any $\xi:R_\cG\ra \fo$ (where $\fo$ is a finite extension of  $\fo_{\breve E}$).
Since $\Spec R_\cG$ is the Zariski closure of $\Spec R_\cG[\ivtd p]$ in $\Spec (R_\GL\otimes_{\Zpbr}\fo_{\breve E})$, the kernel of (\ref{eqn:QisgActionRing}) is the intersection of the kernel of  (\ref{eqn:QisgActionPt}), which should coincide with the kernel of the natural projection $R_\GL\otimes_{\Zpbr}\fo_{\breve E}\thra R_\cG$ by Proposition~\ref{prop:QisgActionRZ}.
\end{proof}

\begin{proof}[Proof of Proposition~\ref{prop:QisgActionRZ}]
Let $X_\xi$ denote the deformation of $\BX$ over $\fo$ corresponding to $\xi$, and let $\iota_\xi:(X_\xi)_{\fo/p}\dra \BX_{\fo/p}$ denote the unique quasi-isogeny lifting the identity map of $\BX$.
We choose a complete algebraically closed extension $C$ of $K = \Frac(\fo)$, and view $\xi$ as an $\fo_C$-point of $R_\cG$. By (completed) scalar extension of (\ref{eqn:QisgActionPt}) over $\fo_C$, we obtain the following map
\begin{equation}\label{eqn:QisgActionC}
R_\GL \otimes_{\Zpbr}\fo_{\breve E} \ra S_G^\infty\wh\otimes_{\Zpbr}\fo_C = \fo_C[[ x_1^{1/p^\infty},\cdots,x_d^{1/p^\infty}]],
\end{equation}
which is determined by the following properties:
\begin{enumerate}
\item The universal deformation of $\BX$ pulls back to the base change of $X_\xi$ via $\fo\ra S_G^\infty\wh\otimes_{\Zpbr}\fo_C$.
\item For any open ideal $I\subset S_G^\infty\wh\otimes_{\Zpbr}\fo_C$ containing $p$ such  that the quotient $(S_G^\infty\wh\otimes_{\Zpbr}\fo_C)/I$ is f-semiperfect, the map $R_\GL\ra (S_G^\infty\wh\otimes_{\Zpbr}\fo_C)/I$, induced by (\ref{eqn:QisgActionC}) corresponds to the following quasi-isogeny
\[\xymatrix@1{
(X_\xi)_{(S_G^\infty\wh\otimes_{\Zpbr}\fo_C)/I} \ar@{-->}[r]^-{\iota_\xi}&
\BX_{(S_G^\infty\wh\otimes_{\Zpbr}\fo_C)/I} \ar@{-->}[r]^-{\gamma^\univ}&
\BX_{(S_G^\infty\wh\otimes_{\Zpbr}\fo_C)/I}
},\]
where $\gamma^\univ \in \Qisg^\circ_G(\BX)((S_G^\infty\wh\otimes_{\Zpbr}\fo_C)/I)$ is the tautological point.
\end{enumerate}
The quasi-isogeny $\gamma^\univ\circ\iota_\xi$ preserves the tensors $(s_\alpha)$ in the sense of Definition~\ref{def:TenHom}.

Let $S_{G,\xi}\subset S_G^\infty\wh\otimes_{\Zpbr}\fo_C$ denote the image of $R_\GL$ by (\ref{eqn:QisgActionC}). Since  $S_{G,\xi}$ is $p$-torsion free by construction, it suffices to show that for any finite extension $\fo'$ of $\fo$, any map $R_\GL\thra S_{G,\xi} \ra \fo'$ factors through $R_\cG$.

For $\xi':S_{G,\xi}\ra \fo'$, let us also choose $\fo'\hra \fo_C$, where $C$ is the fixed complete algebraically closed extension of $\Frac(\fo)$. We also choose and a map $\xi'_{\fo_C}:\Spf\fo_C\ra\Qisg^\circ_G(\BX)$ over $\xi'$.
Let $X_{\xi'}$ denote the deformation of $\BX$ over $\fo'$ corresponding to $\xi'$. By construction, we have $X_{\xi',\fo_C}=X_{\xi,\fo_C}$, where $\xi:R_\cG\ra \fo$ is the fixed $\fo$-point in the statement of Proposition~\ref{prop:QisgActionRZ}. Furthermore, the following maps of $F$-isocrystals
\[
s_\alpha:\triv\ra\DD(X_{\xi',\fo'/p})^\otimes[\ivtd p]\text{ and } s_\alpha:\triv\ra\DD(X_{\xi,\fo/p})^\otimes[\ivtd p]
\]
coincide after the base change over $\fo_C/p$.

Now we want to show that $\xi':R_\GL\ra\fo'$ factors through $R_\cG$ by applying Proposition~\ref{prop:KPEtaleCycles}. The condition on the Hodge filtration (Proposition~\ref{prop:KPEtaleCycles}(\ref{prop:KPEtaleCycles:HodgeFil})) is straightforward, as it can be checked after the base change over $C$, and $\xi$ satisfies Proposition~\ref{prop:KPEtaleCycles}(\ref{prop:KPEtaleCycles:HodgeFil}).

It remains to verify the condition on \'etale tensors (Proposition~\ref{prop:KPEtaleCycles}(\ref{prop:KPEtaleCycles:Torsor})). Note that we have a natural $\Zp$-linear isomorphism
\[T(X_\xi) = T(X_{\xi'})\]
induced by the identification $X_{\xi,\fo_C} = X_{\xi',\fo_C}$. Furthermore, this identification preserves the tensors $(s_{\alpha,\et})$. Indeed, the \'etale tensors $(s_{\alpha,\et})\subset T(X_{\xi'})^\otimes[\ivtd p]$ are determined by the maps of $F$-isocrystals $s_\alpha:\triv\ra\DD(X_{\xi',\fo_C/p})^\otimes[\ivtd p]$ and the Hodge filtration for $X_{\xi',\fo_C}$; indeed, this datum determines the tensors of a certain vector bundle respecting the modification (associated to $X_{\xi',\fo_C}$), which determines the \'etale tensors by extending the comparison map \cite[Corollary~5.1.2]{ScholzeWeinstein:RZ} to tensors of vector bundles with modifications; \emph{cf.} \cite[Theorem~6.1]{FarguesFontaine:VB} and the subsequent Remark.
Since $(T(X_\xi),(s_{\alpha,\et}))$ satisfies Proposition~\ref{prop:KPEtaleCycles}(\ref{prop:KPEtaleCycles:Torsor}), the same holds for $(T(X_{\xi'}),(s_{\alpha,\et}))$.
\end{proof}

\begin{rmksub}\label{rmk:QisgActionRZ}
Assume that $\cG$ is a reductive group over $\Zp$, and $(\BX,(s_\alpha))$ comes from a mod~$p$ point of some integral canonical model of Hodge-type Shimura varieties (with hyperspecial level structure at $p$). Then one can define Hodge-type Rapoport-Zink spaces (\emph{cf.} \cite{Kim:RZ}, \cite{HowardPappas:GSpin}) and Hodge-type Igusa towers (\emph{cf.} \cite{Hamacher:ShVarProdStr}), and it is not difficult to deduce from Theorem~\ref{thm:QisgActionRZ} that the full tensor-preserving self quasi-isogeny group $\Qisg_G(\BX)$ acts on Hodge-type Rapoport-Zink spaces and Igusa towers. Indeed, we have $\Qisg_G(\BX) = \Qisg^\circ_G(\BX)\rtimes \underline{J_b(\Qp)}$, and the action of $J_b(\Qp)$ on these spaces are already defined.
\end{rmksub}

\begin{defnsub}\label{def:QisgOrbits}
Let $\fo$ be a finite extension of $\fo_{\breve E}$, and assume that $\xi: \Spf\fo\ra\Spf R_\cG$ satisfies Proposition~\ref{prop:KPEtaleCycles}(\ref{prop:KPEtaleCycles:Torsor}). Then let $\Qisg^\circ_G(\BX)_\xi$ denote the $\Qisg^\circ_G(\BX)$-orbit of $\xi$; more precisely, we define $\Qisg^\circ_G(\BX)_\xi:=\Qisg^\circ_G(\BX)\times_{\Spf\Zpbr}\Spf\fo$, viewed as a formal scheme over $R_\cG$ via the map defined in Proposition~\ref{prop:QisgActionRZ}. We similarly define a formal scheme $\Qisg^\circ(\BX)_\xi:=\Qisg^\circ(\BX)\times_{\Spf\Zpbr}\Spf\fo$ over $R_\GL$.
\end{defnsub}

We need the following lemma in the next section:
\begin{lemsub}\label{lem:QisgG}
In the setting of Definition~\ref{def:QisgOrbits}, let $C$ be any algebraically closed complete extension of $\Frac(\fo)$. Then, $\xi'\in\Qisg^\circ(\BX)_\xi(\fo_C)$ lies in $\Qisg^\circ_G(\BX)_\xi(\fo_C)$ if and only if $\xi'$ defines an $\fo_C$-point of $R_\cG$. 
\end{lemsub}
\begin{proof}
By Proposition~\ref{prop:QisgActionRZ}, any $\xi'\in\Qisg^\circ_G(\BX)_\xi(\fo_C)$ defines an $\fo_C$-point of $R_\cG$. Conversely, let us assume that $\xi'\in\Qisg^\circ(\BX)_\xi(\fo_C)$ defines an $\fo_C$-point of $R_\cG$, and show that $\xi'\in \Qisg^\circ_G(\BX)_\xi(\fo_C)$. 

Let $\Spf S\subset \Spf R_\GL$ and $\Spf S_\cG\subset \Spf R_\cG$ respectively denote the smallest closed formal subschemes  which the structure morphisms $\Qisg^\circ(\BX)_\xi \ra\Spf R_\GL$ and $\Qisg^\circ_G(\BX)_\xi \ra\Spf R_\cG$ factor through. We want to show that the natural injective map $(\Spf S_\cG)(\fo_C)\ra (\Spf S\otimes_{\Zpbr}R_\cG)(\fo_C)$ is bijective. By the Zariski density consideration, it suffices to show that if $\xi'\in (\Spf S\otimes_{\Zpbr}R_\cG)(\fo_C)$ is defined over some finite extension $\fo'$ of $\fo$, then $\xi'$ lies in the image of $(\Spf S_\cG)(\fo_C)$. For such $\xi'$, we use the same letter $\xi'$ to denote the $\fo'$-point $\xi':R_\cG\ra\fo'$ descending the $\fo_C$-point.


Let $X_{\xi'}$ denote the $p$-divisible group over $\fo'$ corresponding to $\xi'$. Then we have a unique quasi-isogeny
\[
\iota_{\xi'}:X_{\xi',\fo'/p} \dra \BX_{\fo'/p},
\]
lifting the identity map on $\BX$. And as $\xi'$ is an $\fo'$-point of $R_\cG$, the quasi-isogeny $\iota_{\xi'}$ is tensor-preserving. Therefore, the quasi-isogeny
\[
\iota_{\xi'}\iv\circ\iota_\xi:X_{\xi,\fo'/p} \dra X_{\xi',\fo'/p}
\]
is tensor-preserving. Finally, using $X_{\xi',\fo_C} = X_{\xi,\fo_C}$ (which comes from the fact that $\xi'\in\Qisg^\circ(\BX)(\fo_C)$), it follows that the pull back of $\iota_{\xi'}\iv\circ\iota_\xi$ over $\fo_C/p$ is a tensor-preserving self quasi-isogeny of $X_{\xi, \fo_C/p}$, hence defines an element $\gamma_{\xi}\in\Qisg^\circ_G(\BX)(\fo_C)$. Furthermore, by construction we have $\gamma_{\xi'}\cdot \xi = \xi'$, where $\gamma_{\xi'}\cdot \xi$ refers to the natural action of $\gamma_{\xi'}\in\Qisg^\circ(\BX)(\fo_C)$ on $\xi\in R_\GL(\fo_C)$. This shows that $\xi'\in\Qisg^\circ_G(\BX)_\xi(\fo_C)$, as we have claimed.
\end{proof}
%
%
%

\section{Almost product structure in  Kisin-Pappas deformation rings}\label{sec:APS}
Throughout this section, we set $\bar R_\GL:= R_\GL/p R_\GL$ and $\bar R_\cG:=R_\cG/\m_{\breve E}R_\cG$. For any ring $R$ of characteristic~$p$, we let $R^{p^{-\infty}}$ denote the perfection of $R$. If $R$ is a complete local noetherian ring of characteristic $p$, then we write $\wh R^{p^{-\infty}}$ for the $\m_R$-adic completion of $R^{p^{-\infty}}$.  We use the similar notation for  schemes and  formal schemes of characteristic~$p$.
\subsection{Central leaves}

For any geometric point $\bar x:\Spec\kappa\ra \Spec \bar R_\cG$ (with $\kappa$ algebraically closed),
let $X_{\bar x}$ denote the fibre of the universal deformation of $p$-divisible groups over $R_{\cG}$. Then the Dieudonn\'e display $(M_{\cG,\bar x}, M_{\cG,\bar x,1},\Psi)$ of $X_{\bar x}$ can be explicitly described as follows: we have $M_{\cG,\bar x} = W(\kappa)\otimes_{\bar x, \Zink(R_\cG)}M_\cG$,  $M_{\cG,\bar x,1}$ is defined so that its associated Hodge filtration is the fibre of the Hodge filtration associated to $(M_\cG, M_{\cG,1})$, and $\Psi:\wt M_{\cG,\bar x,1}\riso M_{\cG,\bar x}$ is the fibre of $\Psi:\wt M_{\cG,1}\riso M_\cG$. We also have the fibre $(s_{\alpha,\bar x})\subset M_{\cG,\bar x}^\otimes$ of the tensors  $(s_\alpha)\subset M_\cG^\otimes$.
\begin{propsub}\label{prop:OortFoliation}
There exists a reduced closed subscheme of $\gC_\cG\subset \Spec \bar R_\cG$, also denoted by $\gC_\cG^{\db b}$, such that a geometric point $\bar x:\Spec\kappa\ra \Spec \bar R_\cG$ (with $\kappa$ algebraically closed) factors through $\gC_\cG$ if and only if we have $X_{\bar x}\cong \BX_\kappa$. Furthermore, if $\bar x$ defines a geometric point of $\gC_\cG$, then  there exists an isomorphism of Dieudonn\'e displays $M_{\cG,\bar x}\cong W(\kappa)\otimes_{\Zpbr} \bM$ sending $(s_{\alpha,\bar x})$ to  $(1\otimes s_\alpha)\subset W(\kappa)\otimes_{\Zpbr}\bM^\otimes$.
\end{propsub}
\begin{proof}
The proof is completely analogous to the proof of \cite[Proposition~2.14]{Hamacher:DeforSpProdStr}.
Since $\bar R_\cG$ is excellent (as it is a complete local noetherian ring), Oort \cite[Proposition~2.2]{Oort:Foliations} constructed the reduced closed subscheme $\gC_\cG\subset \Spec \bar R_\cG$ such that a geometric point $\bar x$ lands in $\gC_\cG$ if and only if $X_{\bar x}$ is isomorphic to $\BX$. It remains to show that the isomorphism can be chosen to preserve the tensors.

Let $X_{\gC_\cG}$ denote the restriction of the universal deformation over $\gC_\cG$. Recall that for each $n$ there exists a scheme $\gC_{\cG,n}$ finite and surjective over $\gC_\cG$ such that $X_{\gC_{\cG,n}}[p^n]\cong \BX[p^n]_{\gC_{\cG,n}}$. Therefore, over the perfection $\mathfrak{D}$ of $\varprojlim \gC_{\cG,n}$ we have $X_{\mathfrak D} \cong \BX_{\mathfrak{D}}$.

For any open affine subscheme $\Spec R\subset \mathfrak{D}$, we can define a ``Dieudonn\'e display'' $(M_R, M_{R,1},\Psi)$ over $R$, obtained as the base change of $M_\cG$; in other words, $M_R:=W(R)\otimes_{\Zink(R_{\cG})}M_\cG $, and the rest of the datum is defined so that the Hodge filtration and $\Psi$ are compatible with the scalar extension. Then we get the $\Psi$-tensors $(s_{\alpha,R})\subset M_R^\otimes$ by the scalar extension of $(s_\alpha)\subset M_\cG^\otimes$. 
Since $X_{\mathfrak{D}}\cong \BX_{\mathfrak{D}}$, we have a $\Psi$-equivariant isomorphism $M_R\cong W(R)\otimes_{\Zpbr}\bM$ lifting the identity map on $\bM$. So it follows from \cite[Lemma~3.9]{RapoportRichartz:Gisoc} that $(s_{\alpha,R})$ coincide with the scalar extensions of $(s_\alpha)\subset \bM^\otimes$.

Let $\kappa$ be an algebraically closed extension of $\Fpbar$. Then  any $\kappa$-valued  point $\bar x$ of $\gC_\cG$ can be lifted to $\mathfrak{D}$ without increasing $\kappa$ (since $\mathfrak{D}$ is pro-finite surjective over $\gC_\cG$. Choosing such a lift (and an open affine neighbourhood of it), we obtain a tensor-preserving isomorphism $M_{\cG,\bar x} \cong W(\kappa)\otimes_{\Zpbr}\bM$. This concludes the proof.
\end{proof}

By applying Proposition~\ref{prop:OortFoliation} to $R_\GL$, we obtain $\gC_\GL \subset\Spec \bar R_\GL$. It is known that
\begin{equation}\label{eqn:CenLeafGL}
\gC_\GL  \cong \Spec \Fpbar [[x_1,\cdots,x_{d'}]].
\end{equation}
Here, $d' = \langle 2\rho',\nu_{[b]}\rangle$ where $2\rho'$ is the sum of positive roots of $\GL(\Lambda)$, and $\nu_{[b]}$ is the dominant representative of the conjugacy class of $\nu_b$ where $b\in  \GL(\Lambda)(\Qpbr)$ is obtained from the Dieudonn\'e module of $\BX$. (Indeed, if $\BX$ is completely slope divisible, then one even has an explicit description of coordinates; \emph{cf.} \cite[\S7]{Chai:HeckeSiegel}.\footnote{See also \cite[\S3.2]{Hamacher:DeforSpProdStr} for the statement and argument which works for more general connected reductive groups over $\Zp$ instead of $\GL(\Lambda)$.}, and the general case follows from this special case.
) Recall from Proposition~\ref{prop:Qisg} that $d'=\dim\Qisg^\circ(\BX)_{\Fpbar}$.

\begin{thmsub}\label{thm:CenLeaf}
Assume that there exists $\xi: \Spf\fo\ra\Spf R_\cG$ that satisfies Proposition~\ref{prop:KPEtaleCycles}(\ref{prop:KPEtaleCycles:Torsor}).\footnote{As explained in Remark~\ref{rmk:NonNeutral}, this assumption can be arranged if $R_\cG$ came from some integral model of Shimura varieties constructed in \cite{KisinPappas:ParahoricIntModel}. 
} 
Let $\wh\gC_\cG^{p^{-\infty}}$ be the formal completion of the perfection of $\gC_\cG$. Then $\wh\gC_\cG^{p^{-\infty}}$ is the $\Qisg^\circ_G(\BX)_{\Fpbar}$-orbit of the closed point of $\Spf \bar R_\cG$. In other words, there exists an isomorphism $\Qisg^\circ(\BX)_{\Fpbar} \cong \wh\gC_\cG^{p^{-\infty}}$ which fits in the following commutative diagram
\[
\xymatrix{
 \Qisg^\circ_G(\BX) \times_{\Spf\Zpbr}\Spec\Fpbar \ar@{^{(}->}[r] \ar[d]_{\cong}& \Qisg^\circ_G(\BX)_{\Fpbar} \times \Spf \bar R_\cG \ar[d] \\
\wh\gC_\cG^{p^{-\infty}} \ar[r]  & \Spf \bar R_\cG
},\]
where the right vertical map is defined by Theorem~\ref{thm:QisgActionRZ}.
\end{thmsub}
\begin{proof}
Consider the map $\Qisg^\circ_G(\BX)_{\Fpbar}\ra \Spf \bar R_\cG$ induced by the natural action of $\Qisg^\circ_G(\BX)_{\Fpbar}$ on the  closed point (i.e., the composition of the top horizontal arrow and the right vertical arrow in the diagram of the statement). By Proposition~\ref{prop:OortFoliation}, it factors through the completion of $\gC_\cG$ so in turn it factors through $\wh\gC_\cG^{p^{-\infty}}$ by perfectness of $\Qisg^\circ_G(\BX)_{\Fpbar}$.
Therefore, we obtain a natural map
\begin{equation}\label{eqn:CenLeafOrbit}
\Qisg^\circ_G(\BX)_{\Fpbar}\ra \wh\gC_\cG^{p^{-\infty}}
\end{equation}
(not proven to be an isomorphism yet), which fits in the above  commutative diagram.

Let us first handle the special case when $\cG = \GL(\Lambda)$ and $\BX$ is completely slope divisible. Then  the universal deformation $X_{\gC_\GL}$ restricted to $\gC_\GL$ is completely slope divisible (\emph{cf.} \cite[\S2.4.2]{Mantovan:Thesis}) and the associated grading of the slope filtration is isomorphic to $\BX_{\gC_\GL}$ (\emph{cf.} \cite[Lemma~3.4]{Mantovan:Thesis}). It follows that we have an isomorphism $X_{\wh\gC_\GL^{p^{-\infty}}}\cong \BX_{\wh\gC_\GL^{p^{-\infty}}}$, since the slope filtration of a completely slope divisible $p$-divisible group canonically splits over a perfect base. Therefore, the natural map $\wh\gC_\GL^{p^{-\infty}}\ra\Spf R_\GL$ corresponds to a quasi-isogeny
\[
X_{\wh\gC_\GL^{p^{-\infty}}}\cong \BX_{\wh\gC_\GL^{p^{-\infty}}} \dra \BX_{\wh\gC_\GL^{p^{-\infty}}} ,
\]
which corresponds to a map $\wh\gC_\GL^{p^{-\infty}}\ra\Qisg^\circ(\BX)_{\Fpbar}$. This gives the inverse of (\ref{eqn:CenLeafOrbit}).

Now if we retain the assumption that $\cG = \GL(\Lambda)$ and allow $\BX$ to be any $p$-divisible group, then there exists a quasi-isogeny $\gamma:\BX\dra\BX'$ for some completely slope divisible $p$-divisible group $\BX'$ over $\Fpbar$. Note that $\gamma$ induces an isomorphism $\Qisg^\circ(\BX)\cong \Qisg^\circ(\BX')$ defined by the ``conjugation by $\gamma$''. If we let $\gC_\GL'$ denote the central leaf in the universal deformation ring of $\BX'$, then by \cite[Proposition~4.2(2)]{Hamacher:DeforSpProdStr} we have a $\Qisg^\circ(\BX)$-equivariant isomorphism $\wh\gC_\GL^{p^{-\infty}}\cong \wh\gC_\GL^{\prime p^{-\infty}}$. This proves the case when $\cG=\GL(\Lambda)$.

Now, let us handle the general case. By the proof of Proposition~\ref{prop:OortFoliation}, $\gC_\cG$ is the underlying reduced scheme of $\gC_\GL\times_{\Spec R_\GL}\Spec R_\cG$. From the $\GL(\Lambda)$-case,  the natural map $\Qisg^\circ_G(\BX)_{\Fpbar}\ra \wh\gC_\cG^{p^{-\infty}}$ (\ref{eqn:CenLeafOrbit}) is a closed immersion. Let us first show that this is an isomorphism.

We fix $\xi:R_\cG\ra\fo$ (for some finite extension $\fo$ of $\fo_{\breve{E}}$) that satisfies Proposition~\ref{prop:KPEtaleCycles}(\ref{prop:KPEtaleCycles:Torsor}). Then the unique $\fo$-lift of $\wh\gC_\GL^{p^{-\infty}}$ can be identified with $\Qisg^\circ(\BX)_\xi$ (\emph{cf.} Definition~\ref{def:QisgOrbits}), and this lift respects the map to $\Spf R_\GL$. Let $\mathfrak{D}_{\cG,\xi}\subset \Qisg^\circ(\BX)_\xi$ denote the unique $\fo$-flat formal closed subscheme with special fibre $\wh\gC_\cG^{p^{-\infty}}$. By construction, we have a natural closed immersion
\[
\Qisg^\circ_G(\BX)_\xi \hra \mathfrak{D}_{\cG,\xi}.
\]
Furthermore, this map induces a bijection on the set of $\fo_C$-points for any algebraically closed complete extension $C$ of $\Frac(\fo)$ by Lemma~\ref{lem:QisgG}, so it has to be an isomorphism. (To see this, we may replace $\Qisg^\circ_G(\BX)_\xi$ and $\mathfrak{D}_{\cG,\xi}$ with the smallest formal closed subschemes of $\Spf R_\cG$ which the structure morphism factors through, and conclude by Zariski density.) By reducing this isomorphism modulo $\m_\fo$, we show that (\ref{eqn:CenLeafOrbit}) is an isomorphism. This shows the theorem  in general.
\end{proof}

\begin{rmksub}
Theorem~\ref{thm:CenLeaf} does not force $\gC_\cG$ to be isomorphic to a formal spectrum of a formal power series ring over $\Fpbar$. It seems plausible to show that $\gC_\cG$ is formally smooth by generalising the construction of ``generalised Serre-Tate coordinates'' (\emph{cf.} \cite[\S3.2]{Hamacher:DeforSpProdStr}). Later in this paper, we show the formal smooth of $\gC_\cG$ in the case when $R_\cG$ arises from the completed local ring of some integral model of parahoric-level Hodge-type Shimura variety (constructed by Kisin and Pappas \cite{KisinPappas:ParahoricIntModel}); \emph{cf.} Corollary~\ref{cor:CenLeaf}.
\end{rmksub}


\subsection{Product structure of the Newton stratum}
Let $\gN_\GL\subset \Spec \bar R_\GL$ be the closed Newton stratum; i.e., $\gN_\GL$ is defined by the property that $\bar x:\Spec \kappa \ra \Spec \bar R_\GL$ factors through $\gN_\GL$ if and only if $X_{\bar x}$ is isogenous to $\BX_\kappa$. We also define the isogeny leaf $\gJ_\GL\subset \Spec \bar R_\GL$; i.e., the maximal closed subset (viewed as a reduced subscheme) where there exists a quasi-isogeny $X_{\gJ_\GL} \dra \BX_{\gJ_\GL}$ over $\gJ_\GL$ (not just over the formal completion of $\gJ_\GL$) which lifts the identity map on $\BX$. Identifying $R_\GL$ as a completed local ring of some Rapoport-Zink space,  $\gJ_\GL$ is the spectrum of the completed local ring of the underlying reduced scheme of the Rapoport-Zink space. Note that $\gJ_\GL\subset \gN_\GL$.

Let us first recall the almost product structure of the Newton stratum \cite[\S4.1]{Hamacher:DeforSpProdStr} in our language. Let $X_{\gJ_\GL}$ denote the restriction of the universal deformation, and we have a quasi-isogeny $\iota:X_{\gJ_\GL}\dra \BX_{\gJ_\GL}$.
Then the natural $\Qisg^\circ(\BX)$-action on $\Spf R_\GL$ induces the following map
\begin{equation}\label{eqn:APSviaOrbitGL}
\Qisg^\circ(\BX)_{\Fpbar}\times \wh\gJ^{p^{-\infty}}_\GL \ra
\Spf \bar R_\GL.
\end{equation}
Identifying $R_\GL$ as a completed local ring of some Rapoport-Zink space, the map (\ref{eqn:APSviaOrbitGL}) is given by the following quasi-isogeny over $\Qisg^\circ(\BX)_{\Fpbar}\times \wh\gJ^{p^{-\infty}}_\GL$
\begin{equation}\label{eqn:APSQisgGL}
\xymatrix@1{
\pr_2^*X_{\wh\gJ^{p^{-\infty}}_\GL} \ar@{-->}[r]^-{\pr_2^*(\iota)} &\BX_{\Qisg^\circ(\BX)_{\Fpbar}\times \wh\gJ^{p^{-\infty}}_\GL} \ar@{-->}[r]^-{\pr_1^*(\gamma^u)} & \BX_{\Qisg^\circ(\BX)_{\Fpbar}\times \wh\gJ^{p^{-\infty}}_\GL},
}\end{equation}
where $\pr_i$ (with $i=1,2$) is the $i$th projection from $\Qisg^\circ(\BX)_{\Fpbar}\times \wh\gJ^{p^{-\infty}}_\GL$.

The following can be read off from the work of Hamacher:
\begin{propsub}\label{prop:APSGL}
The map (\ref{eqn:APSviaOrbitGL}) induces an isomorphism
\[
\Qisg^\circ(\BX)_{\Fpbar}\times \wh\gJ^{p^{-\infty}}_\GL \riso \wh\gN^{p^{-\infty}}_\GL.
\]
Furthermore, it coincides with the completion of the natural isomorphism $\pi_\infty:\gC_\GL^{p^{-\infty}}\times \gJ_\GL^{p^{-\infty}} \riso \gN_\GL^{p^{-\infty}}$ in \cite[Corollary~4.4]{Hamacher:DeforSpProdStr} via the identification $\Qisg^\circ(\BX)_{\Fpbar}\cong \wh\gC_\cG^{p^{-\infty}}$ given by Theorem~\ref{thm:CenLeaf}.
\end{propsub}
\begin{proof}
Note that the map (\ref{eqn:APSviaOrbitGL})  factors through the closed Newton stratum; indeed, since $\Qisg^\circ(\BX)_{\Fpbar}\times \wh\gJ^{p^{-\infty}}_\GL$ is an affine formal scheme we can view the $p$-divisible group  $\pr_2^*X_{\wh\gJ^{p^{-\infty}}_\GL}$ over the affine scheme algebraising its base formal scheme, and only the geometric fibres of $X_{\gJ_\GL}$ occurs there. Now, by perfectness, we see that the map (\ref{eqn:APSviaOrbitGL})  factors through $\wh\gN^{p^{-\infty}}_\GL$.

Hamacher showed that there exists a natural isomorphism $\pi_\infty:\gC_\GL^{p^{-\infty}}\times \gJ_\GL^{p^{-\infty}} \riso \gN_\GL^{p^{-\infty}}$; \emph{cf.} \cite[Corollary~4.4]{Hamacher:DeforSpProdStr}. Unwinding its proof and identifying $\wh\gC^{p^{-\infty}}_\GL\cong\Qisg^\circ(\BX)_{\Fpbar}$ (\emph{cf.} Theorem~\ref{thm:CenLeaf}), one obtains that our the map $\Qisg^\circ(\BX)_{\Fpbar}\times \wh\gJ^{p^{-\infty}}_\GL \riso \wh\gN^{p^{-\infty}}_\GL$ coincides with the map induced by the isomorphism $\pi_\infty$
 on the formal completions.
\end{proof}

To define the Newton stratification on $\Spec \bar R_\cG$, we need the following lemma.
\begin{lemsub}\label{lem:GIsoc}
For brevity, we write $S:=\bar R_\cG^{p^{-\infty}}$.
Then there exists an exact faithful $\otimes$-functor
\[
\Rep_{\Qp}G \ra F\textrm{-}\isoc_{S}
\]
sending $\Lambda[\ivtd p]$ to $(W(S) \otimes_{\Zink(R_\cG)} M_\cG[\ivtd p],\Psi)$, which depends only on $[b]$ up to  isomorphism.
\end{lemsub}
\begin{proof}
By construction, we have a tensor-preserving isomorphism $\Zink(R_\cG)\otimes_{\Zp}\Lambda\riso M_\cG$. Therefore, $\Psi:\sigma^* M_\cG[\ivtd p]\ra M_\cG[\ivtd p]$ defines an element $b_\Psi\in G(\Zink(R_\cG)[\ivtd p])$. By viewing  $b_\Psi\in G(W(S)[\ivtd p])$, we obtain an $F$-isocrystal $(W(S)\otimes_{\Zp}V,b_\Psi\sig)$ over $\Spec S$ to each algebraic $G$-representation $V$ over $\Qp$. Since $S$ is perfect, this gives the desired exact faithful $\otimes$-functor. Now, a different choice tensor-preserving isomorphism $\Zink(R_\cG)\otimes_{\Zp}\Lambda\riso M_\cG$ modifies $b_\Psi$ up to $\sigma$-$G(\Zink(R_\cG)[\ivtd p])$ conjugacy, so it does not affect the resulting $\otimes$-functor up to isomorphism.
\end{proof}

By the previous lemma, we can apply \cite[Theorem~3.6]{RapoportRichartz:Gisoc} to obtain the Newton stratification on $\Spec \bar R_\cG^{p^{-\infty}}$. Since $\Spec \bar R_\cG^{p^{-\infty}}$ is homeomorphic to $\Spec \bar R_\cG$, we may regard it as a stratification on $\Spec \bar R_\cG$.

Let $\gN_\cG\subset \Spec \bar R_\cG$ be the closed Newton stratum. We define the isogeny leaf $\gJ_\cG\subset \Spec \bar R_\cG$ as the reduced intersection of $\gJ_\GL$ and $\Spec \bar R_\cG$. We also write $\wh\gN_\cG^{p^{-\infty}}$ and $\wh\gJ_\cG^{p^{-\infty}}$ denote the formal completions of the perfections.
\begin{thmsub}\label{thm:APS}
The isomorphism $\Qisg^\circ(\BX)_{\Fpbar}\times \wh\gJ^{p^{-\infty}}_\GL \riso \wh\gN^{p^{-\infty}}_\GL$ in Proposition~\ref{prop:APSGL} restricts to an isomorphism
\[
\Qisg^\circ_G(\BX)_{\Fpbar}\times \wh\gJ^{p^{-\infty}}_\cG \riso \wh\gN^{p^{-\infty}}_\cG.
\]
\end{thmsub}

Before we prove this theorem, let us record the following immediate corollary:
\begin{corsub}\label{cor:APS}
The natural isomorphism $\pi_\infty:\gC_\GL^{p^{-\infty}}\times \gJ_\GL^{p^{-\infty}} \riso \gN_\GL^{p^{-\infty}}$ in \cite[Corollary~4.4]{Hamacher:DeforSpProdStr} restricts to the natural isomorphism $\pi_\infty:\gC_\cG^{p^{-\infty}}\times \gJ_\cG^{p^{-\infty}} \riso \gN_\cG^{p^{-\infty}}$, and its completion recovers the isomorphism in Theorem~\ref{thm:APS}.
\end{corsub}

\begin{proof}[Proof of Theorem~\ref{thm:APS}]
The map is clearly a closed immersion, so it suffices to show that (after algebraising the formal schemes) the image is dense. Since the set of $1$-dimensional points in $\gN_\cG$ is dense, it suffices to show that any map $\xi:\Spec \Fpbar[[t]]\ra\gN_\cG$ factors through the image of $\gC_\GL^{p^{-\infty}}\times \gJ_\GL^{p^{-\infty}} $.

We set $R:=\Fpbar[[t^{p^{-\infty}}]]$, and view $\xi$ also as $\Spf R\ra\wh\gN_\cG^{p^{-\infty}}$. By Proposition~\ref{prop:APSGL}, there exist a map $\gamma_\xi:\Spf R \ra \Qisg^\circ(\BX)$ and a quasi-isogeny $\jmath_\xi:X_\xi \ra \BX_R$ (defined over $\Spec R$), such that $(\gamma_\xi,\jmath_\xi)$ maps to $\xi\in\gN_\GL(R)$. To prove the theorem, we want to show that $\jmath_\xi \in \gJ_\cG(R)$ and $\gamma_\xi \in \Qisg^\circ_G(\BX)(R)$. Since we have  $\jmath_\xi = \gamma_\xi\iv\cdot\xi$, it suffices to show that $\gamma_\xi\in\Qisg^\circ_G(\BX)$  as the $\Qisg^\circ_G(\BX)$-action stabilises $\Spf R_\cG$. 

Let us summarise the setting as follows.
\begin{enumerate}
\item The $R$-point $\jmath_\xi\in\gJ_\GL(R)$ corresponds to the quasi-isogenies 
\[
\jmath_\xi:X_\xi \dra \BX_R
\]
defined over $\Spec R$ (not just over $\Spf R$).
\item The map $\gamma_\xi:\Spf R \ra \Qisg^\circ(\BX)$ corresponds to a collection of quasi-isogenies over $R/t^i$ (for each $i\in\Z_{>0}$)
\[
\gamma_\xi\com i:\BX_{R/t^i} \dra \BX_{R/t^i}.
\]
compatible with respect to the natural projections of the base ring $R/t^i$.
\item The map $\xi:\Spf R\ra\wh\gN^{p^{-\infty}}_\cG$ gives rise to a $p$-divisible group $X_\xi$ over $R$, together with  quasi-isogenies over $R/t^i$ (for each $i\in\Z_{>0}$)
\[
\iota_\xi\com i:X_{\xi, R/t^i} \dra \BX_{R/t^i},
\]
compatible with respect to the natural projections of the base ring $R/t^i$.
\end{enumerate}
Let $M_\xi$ denote the Dieudonn\'e module of $X_\xi$. (Alternatively, $M_\xi$ can also be obtained as the base change from $M_{\cG}$ via the map  $\Zink(R_\cG)\ra W(R)$ induced by $\xi$.) Let  $(s_\alpha)\subset M_\xi^\otimes$ denote the pull back of $(s_\alpha)\subset M_\cG^\otimes$. Similarly, for the f-semiperfect ring $R/t^i$ we have a universal $p$-adic PD hull $\Acris(R/t^i)\thra R/t^i$ (containing $W(R)$ as a subring), and we can view $M_\xi\com i:=\Acris(R/t^i)\otimes_{W(R)}M_\xi$ as a ``Dieudonn\'e module'' of $X_{\xi, R/t^i}$ (\emph{cf.} \cite[\S4]{ScholzeWeinstein:RZ}.) We also obtain the crystalline tensors $(s_\alpha)\subset (M_\xi\com i)^\otimes$ on $X_{\xi,R/t^i}$ by base change.

We claim that $\iota_\xi\com i$  preserves the tensors $(s_\alpha)$. Indeed, the map $R_\cG\ra R/t^i$, induced by $\xi$, can be defined over some artin local subring $\Fpbar[t^{p^{-n}}]/t^i\subset R/t^i$, so the quasi-isogeny $\iota_\xi\com i$ descends over $\Fpbar[t^{p^{-n}}]/t^i$ for some $n\gg0$.
Now, the claim follows since $F$-isocrystals over the spectrum of an artin local ring only depend on the fibre at the closed point (\emph{cf.}  \cite[the proof of Corollary~5.1.2]{dejong:crysdieubyformalrigid}) and the quasi-isogeny $\iota_\xi\com i$ lifts the identity map on $\BX$.

By construction, we have $\gamma_\xi\com i\circ\jmath_{\xi, R/t^i} = \iota_\xi\com i$ for any $i\in\Z_{>0}$.
Therefore, to show $\gamma_\xi\com i\in\Qisg^\circ_G(\BX)(R/t^i)$ for each $i$ (i.e., the induced automorphism of $\Bcris^+(R/t^i)\otimes\bM$ preserves the tensors $(s_\alpha)$), it suffices to show that $\jmath_\xi$ is a tensor-preserving quasi-isogeny over $\Spec R$; i.e., the isomorphism of $F$-isocrystals over $\Spec R$
\begin{equation}\label{eqn:KatzIsotriv}
M_\xi[\ivtd p] \riso W(R)\otimes_{\Zpbr}\bM[\ivtd p],
\end{equation}
induced by $\jmath_\xi$, preserves the tensors $(s_\alpha)$.\footnote{Here, one can give an alternative construction of $\jmath_\xi$ without using Proposition~\ref{prop:APSGL}, as follows. Since $X_\xi$ has \emph{constant} Newton polygon, it is known that there exists a unique isomorphism of $F$-isocrystals $M_\xi[\ivtd p] \riso W(R)\otimes_{\Zpbr}\bM[\ivtd p]$ that reduces to the identity map on $\bM[\ivtd p]$ via $W(R)\thra W(\Fpbar) = \Zpbr$. (The existence of such isomorphism is proved in \cite[Theorem~2.7.4]{Katz:SlopeFil}, and the uniqueness follows from \cite[Lemma~3.9]{RapoportRichartz:Gisoc}). Now, by the Dieudonn\'e theory over $R$ (\emph{cf.} \cite[Corollaire~3.4.3]{Berthelot:PerfectValuationRing}), the above isomorphism of $F$-isocrystals gives rise to a unique quasi-isogeny $X_\xi\dra \BX_R$ of $p$-divisible groups over $\Spec R$, which should recover $\jmath_\xi$ by  uniqueness.}   Indeed, (\ref{eqn:KatzIsotriv}) should be tensor-preserving, since any Frobenius-invariant tensor on a \emph{constant} $F$-isocrystal over $R$ is determined by its special fibre (\emph{cf.} \cite[Lemma~3.9]{RapoportRichartz:Gisoc}), and $\jmath_\xi$ reduces to the identity map on the special fibre $\BX$. This concludes the proof.
\end{proof}

\subsection{Application to integral models of Hodge-type Shimura varieties of parahoric level}
Let $\sS'$ be the scheme over $\Zpbr$ classifying principally polarised abelian varieties of dimension~$g$ with some prime-to-$p$ level structures. Let $\sS$ be a scheme finite and unramified over $\sS'_{\fo_{\breve E}}:=\sS'\times_{\Spec\Zpbr}\Spec\fo_{\breve E}$, where $\fo_{\breve E}$ is a finite extension of $\Zpbr$. Assume that there exists a smooth algebraic subgroup $\cG\subset \GSp_{2g}$ over $\Zp$ which satisfies the following:
\begin{enumerate}
\item The generic fibre $G:=\cG_{\Qp}$ is a reductive group, and $\cG = \cG_x$ is a Bruhat-Tits integral model of $G$ associated to $x\in\cB(G,\Qp)$ (\emph{cf.} Definition~\ref{def:KP}). 
\item 
For any $x\in\sS(\Fpbar)$, there exists an isomorphism of $p$-divisible groups $\BX_b\cong\sA_x[p^\infty]$ for some $b\in G(\Qpbr)$ (where $\sA_x$ is the fibre at $x$ of the universal abelian scheme over $\sS'$), such that $b$ and the embedding $\cG\hra\GSp_{2g}\subset\GL_{2g}$ satisfy the conditions in Definition~\ref{def:LocShDatum}, and  the formal completion $\wh\sS_x\subset \wh\sS'_{x}\times_{\Spf\Zpbr}\Spf\fo_{\breve E}$ coincides with $\Spf R_\cG$ (\emph{cf.} Definition~\ref{def:LocShDatum}) as a closed formal subscheme of the universal deformation space of $\BX_b$.
\end{enumerate}
Note that if $(G,\UH)$ is a Shimura datum that admits an embedding into a Siegel Shimura datum, then the integral model of Shimura varieties with level $\KK = \KK^p\cG(\Zp)$ constructed in \cite{KisinPappas:ParahoricIntModel} gives an example of $\sS$ if $p>2$ does not divide the order of $\pi_1(G^\der)$; \emph{cf.} \cite[Corollary~4.2.4]{KisinPappas:ParahoricIntModel}.

Let us choose a closed point $x\in\sS(\Fpbar)$, and set $\BX:=\sA_x[p^\infty]$ (where $\sA$ is the pull back of the universal abelian scheme over $\sS'$). Let $\Spf R_\GL$ denote the formal scheme classifying $p$-divisible groups deforming $\BX$, and we identify $\wh\sS_x = \Spf R_\cG$ as before. By  \cite[Proposition~2.2]{Oort:Foliations}, there exists a (reduced) locally closed subscheme $\sC_\sS(\BX)\subset \sS_{\Fpbar}$ such that a geometric point $\bar y:\Spec \kappa\ra\sS_{\Fpbar}$ factors through $\sC_\sS(\BX)$ if and only if we have $\sA_{\bar y}[p^\infty]\cong\BX_\kappa$. The following is a corollary of Theorem~\ref{thm:CenLeaf}.

\begin{corsub}\label{cor:CenLeaf}
We write $\sC:=\sC_{\sS}(\BX)$ for brevity.
For any $y\in\sC(\Fpbar)$, the formal completion $\wh\sC_y$ coincides with the formal completion $\wh\gC_\cG$ of the central leaf $\gC_\cG\subset \Spec R_\cG$ (as in Theorem~\ref{thm:CenLeaf}). Furthermore, $\sC$ is smooth of equi-dimension $\langle 2\rho, \nu_{[b]}\rangle$ (using the notation from Proposition~\ref{prop:QEnd}).
\end{corsub}

\begin{proof}
The isomorphism $\wh\sC_y\cong\wh\gC_\cG$ is clear from the definition of $\sC$ and $\gC_\cG$ (by looking at the geometric points of $\sC$ and $\gC_\cG$). Furthermore, $\wh\sC_y$'s are isomorphic for any $y\in\sC(\Fpbar)$, so they have to be formally smooth by the openness of smooth locus of $\sC$. Finally, the dimension of $\sC$ follows from Theorem~\ref{thm:CenLeaf}.
\end{proof}

\begin{rmksub}
If $\sS$ is an integral model of Hodge-type Shimura varieties constructed in \cite{KisinPappas:ParahoricIntModel}, then it is possible to show that there exist crystalline tensors (i.e., maps of $F$-isocrystals) over $\sS_{\Fpbar}^\perf$
\[
s_\alpha:\triv\ra(\DD(\sA[p^\infty])^*)^\otimes[\ivtd p],
\]
such that the pointwise stabiliser of the fibres $(s_{\alpha,x})$ at any $x\in\sS(\Fpbar)$  is isomorphic to $G_{\Qpbr}$; \emph{cf.} \cite[Corollary~3.3.7]{HamacherKim:Mantovan}. Then by repeating the proof of Theorem~\ref{thm:CenLeaf}, one can show that the tensors $(s_\alpha)$ are fibrewise constant on each central leaf $\sC_\sS(\sA_x[p^\infty])$.
\end{rmksub}


\begin{rmksub}\label{rmk:globalAPS}
If $\sS$ is an integral canonical model of Hodge-type Shimura varieties with \emph{hyperspecial} level structure, then Hamacher \cite{Hamacher:ShVarProdStr} constructed a natural tower of finite \'etale coverings of a central leaf generalising Igusa towers considered in  \cite{HarrisTaylor:TheBook,Mantovan:Thesis,Mantovan:Foliation}. We expect that the existence of ``Igusa towers'' and the almost product structure of the Newton strata (in the sense of \cite{Hamacher:ShVarProdStr}) should generalise for integral models of Hodge-type Shimura varieties with \emph{parahoric} level structure constructed in \cite{KisinPappas:ParahoricIntModel}. This question, as well as its cohomological consequence, is considered in \cite{HamacherKim:Mantovan} using the main result of this paper.

Note that just as the case of hyperspecial levels, in order to define various natural group actions on Igusa towers, one needs to understand ``isogeny classes'' of mod~$p$ points of $\sS$; i.e., the existence of a natural map from affine Deligne-Lusztig varieties to $\sS(\Fpbar)$ (\emph{cf.}  \cite[Proposition~1.4.4]{Kisin:LanglandsRapoport}). In \cite{HamacherKim:Mantovan}, we need to assume  the ``parahoric-level generalisation'' of \emph{loc.~cit.}, which is not fully known yet. On the other hand, there is a special case outside the hyperspecial case where such a generalisation can be obtained -- namely, the case when $G$ is residually split using \cite[Theorem~0.2]{HeZhou:ConnCompADL} -- and it seems quite reasonable to believe that the main result of  \cite{HeZhou:ConnCompADL} should hold in more generality.
\end{rmksub}
\bibliography{bib}
\bibliographystyle{amsalpha}

\end{document}